\documentclass[10pt,reqno,final]{amsart}
\usepackage{ifplatform}
\usepackage{epsfig,amssymb,amsmath,version}
\usepackage{amssymb,version,graphicx,fancybox,mathrsfs,multirow}
\usepackage{url,hyperref}
\usepackage[notcite,notref]{showkeys}

\usepackage{subfigure}%,mathabx
\usepackage{color,xcolor}
\usepackage{cases}
\usepackage{mathtools}
\usepackage{wrapfig}

\catcode`\@=11 \theoremstyle{plain}
\@addtoreset{equation}{section}   % Makes \section reset 'equation' counter.

\@addtoreset{figure}{section}
\renewcommand\thefigure{\thesection.\@arabic\c@figure}
\renewcommand{\thefigure}{\arabic{section}.\arabic{figure}}
\newtheorem{thm}{\bf Theorem}

\newtheorem{cor}{\bf Corollary}
\newtheorem{prop}{Proposition}[section]

\newtheorem{lmm}{\bf Lemma}

\newenvironment{lemma}{\begin{lmm}}{\end{lmm}}
\theoremstyle{remark}
\newtheorem{rem}{\bf Remark}[section]
\theoremstyle{definition}
\newtheorem{defn}{\bf Definition}[section]

\numberwithin{table}{section}

\def \rd {{\rm d}}

\def  \Proj {\Pi}
\def  \GProj {\pi}

\def  \RefDomain {I_{\rm ref}}
\newcommand{\bs}[1]{\boldsymbol{#1}}
\newcommand{\bv}[1]{{\rm BV}(#1)}
\newcommand{\bvNorm}[2]{\|D #1\|(#2) }

\def\MD{{\mathcal {D}}}
\renewcommand \wedge \times

\def \rFraSNorm {U_{\hat \theta}^{\alpha}({\omega;\RefDomain})}
\def \rFraSLNorm {{U_{-1+}^{\alpha, m}}({\omega;\RefDomain})}
\def \rFraSRNorm {U^{\alpha,m}_{ 1 -}({\omega;\RefDomain})}

\def \rFraMiSRNorm {U^{(\alpha,m)}_{x_{1}-}({\tilde\omega;I_{1}})}

\newcommand{\CFrD}[2]{{}^{C\!} {\mathcal{ D}}_{#2}^{#1}}

\newcommand{\FraNorms}[3]{U_{#2}^{#1}(#3)}
\newcommand{\FraSpace}[3]{{\mathcal U}_{#2}^{#1}(#3)}

\newcommand{\T} {\top}

\newcommand{\ndg}[1]{| \kern -.25mm \|{#1}| \kern -.25mm \|}
\def\mesh {\mathcal{T}}
\newcommand{\Vertiii}[1]{{\vert\kern-0.25ex\vert\kern-0.25ex\vert #1
		\vert\kern-0.25ex\vert\kern-0.25ex\vert}}

\newcommand{\VertV}[1]{{\vert\kern-0.25ex\vert\kern-0.25ex\vert #1
		\vert\kern-0.25ex\vert\kern-0.25ex\vert}}

\begin{document}
\bibliographystyle{plain}
%\bibliographystyle{unsrt}

%\graphicspath{{./Figures/}}
\baselineskip 13pt

% Gauss--Radau  projection
% Polynomial   singular functions

% (GGR)
\title[Convergence of LDG methods] {
 Optimal convergence of  local discontinuous Galerkin methods for convection-diffusion equations}
\author[
	W. Liu,\, R. Xie,\, L. Wang $\&$\,   Z. Zhang
	]{
		\;\; Wenjie Liu${}^{1}$,\;\;  Ruiyi Xie${}^{1}$,\;\; Li-Lian Wang${}^{2}$  \;\;and\;\;Zhimin Zhang${}^{3}$
		}
	
	\thanks{${}^{1}$School of Mathematics, Harbin Institute of Technology, Harbin 150001, PR China.
The research of the first author was supported by the
National Natural Science Foundation of China (No. 12271128), and the Natural Science Foundation of Heilongjiang Province of China (No. YQ2023A002). Emails: liuwenjie@hit.edu.cn, 22B912009@stu.hit.edu.cn.
The first author is the corresponding author.
		\\
		\indent ${}^{2}$Division of Mathematical Sciences, School of Physical
		and Mathematical Sciences, Nanyang Technological University,
		637371, Singapore. The research of this author is partially supported by Singapore MOE AcRF Tier 1 Grant: RG95/23,  and Singapore MOE AcRF Tier 2 Grant: MOE-T2EP20224-0012. Email: lilian@ntu.edu.sg.
        \\
		\indent ${}^{3}$Department of Mathematics, Wayne State University, Detroit, MI, 48202, USA. Email: ag7761@wayne.edu.}
\begin{abstract}
%In this paper, we obtain the  optimal rates of convergence of the $p$-version local discontinuous Galerkin methods for convection-diffusion equation consider in Castillo, Cockburn, Sch\"{o}tzau and Schwab [Math. Comp., 71(238):455–478, 2002]. We derive explicit and sharper error bounds for Gauss--Radau projections with interior or endpoint singularities in  fractional spaces.
%Our estimates are optimal in the polynomial approximation order $p$ in fractional spaces. The theoretical results are confirmed by numerical experiments.
The $hp$ local discontinuous Galerkin (LDG) method proposed by Castillo et al. [Math. Comp.,~71 (238): 455–478, 2002] has been shown to be an efficient approach for solving convection–diffusion equations. However, theoretical analysis indicates that, for solutions with limited spatial regularity, the method exhibits suboptimal convergence in $p$, suffering a loss of one order, comparing to numerical experiments. The purpose of this paper is to close the gap between theoretical estimates and numerical evidence. This is accomplished by establishing new approximation results for the associated Gauss–Radau projections of functions in suitable function spaces that can optimally characterize the regularity of singular solutions. We show that such a framework arises naturally and enables the study of various types of singular solutions, with full consistency between theoretical analysis and numerical results.
This investigation sheds light on the resolution of the suboptimality in $p$ observed in the literature for several other types of DG schemes in different settings.
\end{abstract}

\maketitle

%\vspace*{-10pt}
\section{Introduction}

Discontinuous Galerkin (DG) finite element methods, first introduced in the early 1970s for solving the neutron transport equation \cite{Lasaint1974}, have evolved into a powerful framework for the numerical solution of a wide range of partial differential equations (PDEs).
The local discontinuous Galerkin (LDG) method of Cockburn and Shu \cite{Cockburn1998SJNAlocal} extends DG techniques to general convection--diffusion operators by introducing suitable auxiliary variables and numerical fluxes, building on earlier DG discretizations of viscous terms such as the scheme of Bassi and Rebay \cite{Bassi1997JCPhigh} for the compressible Navier--Stokes equations.
DG methods are widely employed in practice owing to their flexibility on arbitrarily unstructured meshes, high efficiency in parallel implementation, and their capability to handle complex geometries and interfaces as well as to accommodate general $hp$-adaptivity.

At the level of theoretical analysis, there is a significant difference between the results on optimal convergence rates in terms of the mesh size $h$ and the polynomial degree $p$.
For the former, the theoretical issues of optimal $h$-convergence—and even superconvergence (for $Q_k$)—have long been resolved for triangular meshes with $P_k$ polynomial spaces and tensor-product meshes with $Q_k$ polynomial spaces (see \cite{2015CaoSJNASuperconvergence,2018CaoSJNASuperconvergence,
2016CaoMCSuperconvergence,2014CaoSJNASuperconvergence}). Recently, new progress has also been made for $P_k$ spaces on tensor-product meshes (see \cite{2026CaoMCUnified}).
In contrast, even in the one-dimensional case, the optimal convergence order with respect to the polynomial degree $p$ remains unresolved. The 2002 work of Castillo et al. ~\cite{Castillo2002MC} obtained only sub-optimal results, and since then there has been no new theoretical progress. It is in this context that our work arises.

Our focus is on the prototypical $hp$-DG scheme for convection--diffusion problems proposed by Castillo et al.~\cite{Castillo2002MC}.
This method is built upon a judicious construction of local fluxes along with delicate analysis of the associated Gauss-Radau projections for  the convergence proof. Indeed, it has inspired many follow-ups in both algorithm development and related analysis.
For instance, Perugia and Sch\"otzau \cite{Perugia2002JSC} introduced the $hp$-LDG method for diffusion problems on unstructured meshes with hanging nodes in two and three space dimensions,
together with  further development for low-frequency time-harmonic Maxwell equations \cite{Perugia2003MC$hp$}. Sch\"otzau, Schwab, and Toselli \cite{Schoetzau2002SJNAMixed} constructed mixed $hp$-DGFEMs for incompressible flows and derived a priori $hp$-error estimates for the Stokes problem.
Moreover, Celiker and Cockburn \cite{Celiker2007MCSuperconvergence} studied superconvergence of numerical traces for a broader class of DG-type and hybridized methods for one-dimensional convection--diffusion problems.
Houston, Sch\"otzau, and Wihler~\cite{Houston2007MMMASEnergy} derived $hp$-version a posteriori error bounds in a mesh-dependent energy norm for interior penalty DG discretizations of elliptic boundary-value problems.
As some related work on DG methods for hyperbolic problems, Meng, Shu, and Wu \cite{Meng2016MCOptimal} proved optimal a priori error estimates for upwind-biased fluxes in time-dependent linear transport, while Meng and Ryan \cite{Meng2017NMDiscontinuous} used divided-difference error estimates to obtain superconvergence and postprocessing results for one-dimensional nonlinear scalar conservation laws.  The list is by no means exhaustive.
\begin{comment}
{\color{blue}
% Related superconvergence results for one-dimensional upwind DG methods for %linear
% hyperbolic equations were proved by
% %Cao, Zhang, and Zou
% \cite{2018CaoSJNASuperconvergence,2014CaoSJNASuperconvergence}.
% For one-dimensional linear parabolic equations, Cao and Zhang \cite{2016CaoMCSuperconvergence} established analogous superconvergence properties for alternating flux LDG methods.
In the one-dimensional setting, superconvergence was established for upwind DG discretizations of linear and scalar nonlinear hyperbolic equations, and for LDG discretizations of linear parabolic equations with alternating numerical fluxes; see \cite{2018CaoSJNASuperconvergence,
2016CaoMCSuperconvergence,2014CaoSJNASuperconvergence}.
The superconvergence properties of the DG method for two-dimensional hyperbolic conservation laws  proposed in
\cite{2015CaoSJNASuperconvergence,2026CaoMCUnified}.
}
\end{comment}

However, it has long been known theoretically  that some $hp$-LDG schemes may suffer a loss of convergence order in $p$ for solutions with limited spatial regularity.
In fact, the error analysis in~\cite{Castillo2002MC} shows suboptimal convergence in $p$, with a loss of one order as observed in the numerical tests.
Related $p$-suboptimality is also reported in~\cite{Perugia2002JSC,Perugia2003MC$hp$}, where the error estimates are optimal in the mesh size $h$ but slightly suboptimal in $p$ (losing half an order on general meshes in~\cite{Perugia2002JSC}, and losing half a power of $p$ in~\cite{Perugia2003MC$hp$}).
Other DG formulations exhibit analogous $p$-suboptimal behaviour under limited regularity.
For advection--diffusion--reaction problems, Houston, Schwab, and S\"uli~\cite{Houston2002SJNADiscontinuous} proved bounds that are $hp$-optimal in the hyperbolic regime but $p^{1/2}$-suboptimal in the elliptic case, while Georgoulis et al.~\cite{Georgoulis2010JSCsuboptimality} showed a loss of half an order in $p$ for the $p$-version interior penalty DG method for second-order elliptic boundary value problems under standard Sobolev regularity assumptions.
% Together with the $p$-suboptimality identified for the $hp$-LDG scheme in \cite{Castillo2002MC}, these results indicate that standard Sobolev-based regularity descriptions do not, in general, capture the $p$-version behaviour for nonsmooth solutions.
% Similar effects are also well known for conforming finite elements; see, for example, \cite{Babuska1994SRp,Eriksson1986SJNASome}.

The aim of the present work is to close the theoretical gap between the error estimates in \cite{Castillo2002MC} and  the  numerical evidence therein.
More precisely, we establish a new approximation framework for the Gauss--Radau projections of singular functions in suitable  spaces that can optimally characterize their regularity.  We show that such a framework naturally arises and allows us to study various types of singular solutions (e.g., interior singularities fitting or unfitting the grids) and obtain optimal order convergence in $p$ for the LDG scheme.
Although our analysis is restricted to the one-dimensional setting in \cite{Castillo2002MC}, the proposed framework also sheds light on possible remedies for the observed \(p\)-suboptimality in other settings. These extensions will be reported in future work.

The paper is organized as follows. In Section \ref{LDGEA}, we recall the LDG method for convection-diffusion equations, along with the error estimate results presented in \cite{Castillo2002MC}.
% In Section \ref{sect:Spaces}, we recall the space of functions of bounded variation and present the new fractional spaces.
% The explicit and sharper error estimates for Gauss--Radau projections of functions with singularities are proven in Section \ref{sectmain}.
In Section~\ref{Sect3:Main-Result}, we derive optimal $p$--version error estimates for solutions with endpoint singularities via a fractional-regularity characterization based on Legendre expansion coefficients.
We extend the LDG analysis to interior singularities and distinguish the fitted and unfitted cases in Section~\ref{two cases}.
Finally, in Section \ref{ConcludingRem}, we conclude this paper with some remarks.

\section{The $hp$-LDG scheme and its a priori estimates in \cite{Castillo2002MC}}
\label{LDGEA}
\setcounter{equation}{0}
\setcounter{lmm}{0}
\setcounter{thm}{0}

In this section, we present the $hp$-LDG method for a model convection--diffusion equation studied in \cite{Castillo2002MC} and state the main convergence result therein.
We then discuss its suboptimality for singular solutions, which serves as the motivation for introducing a new framework to obtain optimal estimates in the forthcoming section.

% \subsection{The LDG scheme and its a priori error estimates in  \cite{Castillo2002MC}}
We consider the convection-diffusion equation with the initial condition and  Dirichlet boundary conditions as in \cite{Castillo2002MC}:
\begin{subequations}\label{CDEq}
\begin{align}
u_{t}+(c\; u-d\; u_{x})_{x}=f \quad & \text { in }\;\; Q_{T}:=I \times J, \label{CDEq-1}\\
u(x,0)=u_{\rm ic}(x) \quad & \text { on }\;\; I:=(a, b), \label{CDEq-2}\\
u(a,t)=g_a(t),\quad u(b,t)=g_b(t) \quad & \text { on }\;\; J:=(0, T),  \label{CDEq-3}
\end{align}
\end{subequations}
% \begin{subequations}
% %\begin{dcases}
% \label{CDEq}
% \begin{equation}\label{CDEq-1}
% u_{t}+(c\; u-d\; u_{x})_{x}=g \quad \text { in }\;\; Q_{T}:=(a, b) \times(0, T),
% \end{equation}
% with the initial condition
% \begin{equation}\label{CDEq-2}
% u|_{t=0}=u_{0} \quad \text { on }\;\; \Omega:=(a, b),
% \end{equation}
% and the Dirichlet boundary conditions
% \begin{equation}\label{CDEq-3}
% u(a)=u_{D}(a), \quad u(b)=u_{D}(b) \quad \text { on }\;\; J:=(0, T),
% \end{equation}
% \end{subequations}
%\end{dcases}
where the velocity $c>0$ and diffusion coefficient $d \geq 0.$ Note that in the purely convective case $(d=0)$, only the Dirichlet boundary condition at $x=a$ is imposed.

%\subsection{Related works}
%the discontinuous Galerkin (DG) and the local discontinuous Galerkin (LDG) methods for solving one-dimensional time dependent linear conservation laws and convection-diffusion equations.

To formulate the LDG scheme, we
partition the interval $I$ into the non-overlapping sub-intervals:
$$\mathcal{T}=\big\{I_{j}:=(x_{j-1}, x_{j}):~ j=1, \ldots, M\big\},$$
with the nodes: $a=x_{0}<x_{1}<\cdots<x_{M-1}<x_{M}=b.$
Denote $h_{j}:=|I_{j}|=x_{j}-x_{j-1}$ and
$h:=\max\{h_{j}\}.$
% Given   we partition the given interval $\Omega$ as
% (cf. \cite{Castillo2002MC}):
%
%The weak formulation we are going to use is obtained as follows. First, we introduce the new variable $q:=\sqrt{d} u_{x}$ and the ``flux" function
%$$
%\mbf{h}=(h_{u}, h_{q})^{T}:=(c u-\sqrt{d} q,-\sqrt{d} u)^{T},
%$$
%and rewrite \eqref{CDEq-1}-\eqref{CDEq-3} in the form
%\begin{equation}\label{CDEq-4}
%\begin{array}{ll}
%u_{t}+(h_{u})_{x}=f & \text { in }\;\; Q_{T}, \\
%q+(h_{q})_{x}=0 & \text { in } \;\;Q_{T}, \\
%u|_{t=0}=u_{0} & \text { on }\;\; \Omega, \\
%u(a)=u_{D}(a), \quad u(b)=u_{D}(b) & \text { on }\;\;(0, T) .
%\end{array}
%\end{equation}
 To the mesh $\mathcal{T}$, we associate the so-called broken Sobolev space
$$
H^{1}(I;\mathcal{T}):=\big\{v: I\to \mathbb{R}\,:\, v|_{I_{j}} \in H^{1}(I_{j}),\; j=1, \ldots, M\big\}.
$$
Let $(u, v)_{I_j}:=\int_{I} u(x) v(x) \rd x$ be the inner product of $L^2(I_j),$ and denote
$u(x_{j}^{\pm}):=\lim _{x \to x_{j}^{\pm}} u(x).$
We further define the approximation space of piecewise polynomials:
$$
V_{h}^p=\big\{\phi: I \to \mathbb{R}\,:\,\phi|_{I_{j}} \in \mathcal{P}_{p_{j}},\;  j=1, \ldots, M\big\},
$$
where   $\mathcal{P}_{p}$ denotes the set of  polynomials of degree at most $p$.
The LDG scheme proposed in  \cite{Castillo2002MC}
for  \eqref{CDEq} is to find $u_{h}^{p}(\cdot, t),\, q_{h}^{p}(\cdot, t)\in V_{h}^p$ for  $t\in (0,T)$
such that
\begin{equation}\label{LDG-1}
\begin{aligned}
((u_{h}^{p})_{t}, v)_{I_{j}}-(c\,u_{h}^{p}-\sqrt{d}\, q_{h}^{p}, v_{x})_{I_{j}}+\hat{h}_{u} v\Big|_{x_{j-1}^{+}} ^{x_{j}^{-}} &=(g, v)_{I_{j}}, \\
(q_{h}^{p}, r)_{I_{j}}+\sqrt{d}(u_{h}^{p}, r_{x})_{I_{j}}+\hat{h}_{q} r\Big|_{x_{j-1}^{+}} ^{x_{j}^{-}} &=0, \\
(u_{h}^{p}(\cdot, 0), v)_{I_{j}} &=(u_{\rm ic}, v)_{I_{j}},
\end{aligned}
\end{equation}
for all $v, r \in V_{h}^p$ and $j=1, \ldots, M,$ where
the numerical flux in \cite{Castillo2002MC} was chosen as
\begin{equation}\label{numflux}
(\hat{h}_{u},\hat{h}_{q})^{\T}=
\begin{cases}\big(c\, g_{a}(t)-\sqrt{d}\, q(a^{+}),-\sqrt{d}\, g_{a}(t)\big)^{\T}, & \text { for } j=0, \\[6pt]
 \big(c\, u(x_{j}^{-})-\sqrt{d}\, q(x_{j}^{+}),-\sqrt{d}\, u(x_{j}^{-})\big)^{\T}, & \text { for } j=1, \ldots, M-1, \\[6pt]
\big(\hat{u}(b)-\sqrt{d}\, q(b^{-}),-\sqrt{d}\, g_{b}(t)\big)^{\T}, & \text { for } j=M,
\end{cases}
\end{equation}
with
%\begin{equation}\label{ubsUB}
%\hat{u}(b)=c/2\,( u(b^{-})+g_{b}(t))-\max \{c / 2, \max \{1, p_{M}\} d / h_{M}\}(g_{b}(t)-u(b^{-})).
%\end{equation}
\begin{equation}\label{ubsUB}
\hat{u}(b)=c\, u(b^{-})-\max \{c / 2, \max \{1, p_{M}\} d / h_{M}\}(g_{b}(t)-u(b^{-})).
\end{equation}
Note again that in the purely convective case, $d = 0$, this numerical flux is the standard upwinding flux used by the original DG method.

We recall the a priori error estimates  stated  in \cite[Theorem 3.4]{Castillo2002MC}.
\begin{thm}
\label{Castillo2002}
{\bf (see \cite[Theorem 3.4]{Castillo2002MC}).}
Let $u$ be the solution of \eqref{CDEq} and
$q=\sqrt{d}\, u_x,$ and let $(u_{h}^{p},q_{h}^{p})^{\T}$ be the LDG solution of \eqref{LDG-1}-\eqref{ubsUB}. Assume that $u$ has the regularity
$\|u^{(s+1)}\|_{\mathcal{E}, T}<\infty$ for $s\ge 0,$  where the norm $\|\cdot\| _{\mathcal{E}, T}$ is defined as
$$
\|w\|_{\mathcal{E}, T}:=2 \sup _{0 \leq t \leq T}\|w(\cdot,t)\|+\int_0^T\left\|w_t(\cdot, t)\right\| \rd t+3 \sqrt{d}\left\|w_x\right\|_{Q_T}.
$$
Then we have the error estimate
%$\bs{e}=(u,q)^{\top}-(u_N,q_N)^{\T}:$
% denote the error between the exact solution and the LDG approximation with numerical flux \eqref{numflux} and polynomial degree $p$ on each interval. Then, for completely arbitrary meshes, the error in the energy norm satisfies the inequality
\begin{equation}\label{estimate-of-enorm}
\|(u-u_{h}^{p})(\cdot,T)\|+\|q-q_{h}^{p}\|_{Q_{T}} \leq C(s) \frac{h^{\min \{s, p\}+1}}{\max
\{1,p\}^{s+1}}\big\|u^{(s+1)}\big\|_{\mathcal{E}, T},
\end{equation}
where $p=\min\{p_j\}$ and  $C(s)$ depends on $s$ but is independent of $h,$ $p$, and $u$.
\end{thm}
%The theoretical analysis in leads to error bounds that are suboptimal in terms of $p$ by one order for nonsmooth solutions $x^{\pi} t$ (see \cite[Sec. 5.3]{Castillo2002MC}).

To motivate our analysis, we outline  two main ingredients for the   proof of Theorem \ref{Castillo2002}.
The first is the introduction of the Gauss-Radau projections (see \cite[
(3.1)-(3.2)]{Castillo2002MC}).  Let $\pi^\pm: H^{1}(I;\mathcal{T})\to V_h^p$ be the projections defined by  $\pi^{ \pm} u|_{I_j}=\pi_{p_j}^{\pm} (u|_{I_j})$   on each sub-interval $I_j=\left(x_{j-1}, x_j\right), j=1, \ldots, M,$
where the local projections
 satisfy the following $p_j+1$ conditions:
\begin{equation}\label{7projections}
\begin{split}
 (\pi^{\pm}_{p_j} \omega- \omega, v)_{I_{j}}=0,\quad  & \forall\, v \in \mathcal{P}_{p_{j}-1}(I_{j}),  \\[4pt]
 \pi^{-}_{p_j}  \omega(x_{j})= \omega(x_{j}^{-}),\; \quad  & \pi^{+}_{p_j}  \omega(x_{j-1})= \omega(x_{j-1}^{+}) .
\end{split}
\end{equation}
Then in the analysis, $(\pi^{-} u, \pi^{+} q)$ are applied to $(u,q),$ respectively.
The second is the following important estimate (see \cite[Pages 460-461]{Castillo2002MC}):
\begin{equation}\label{globalbound}
\begin{split}
   \|(u-u_{h}^{p})(\cdot,T)\|  +\|q-q_{h}^{p}\|_{Q_{T}} & \leq  \sqrt{A(T)}
    +\int_{0}^{T} \|(\pi^{-} u_{t}-u_{t})(\cdot, t)\|\, \rd t
    \\
    &\quad +\sup _{0\le t\le T}\|(\pi^{-} u-u)(\cdot, t)\|
    +\|\pi^{+} q-q\|_{Q_{T}},
    \end{split}
\end{equation}
where
\begin{equation}\label{ATcase}
A(T)=\|\pi^{-} u_{\rm ic}-u_{\rm ic}\|^{2}+\|\pi^{+} q-q\|_{Q_{T}}^{2}+\frac{d}{c_{11}(b)}\|(\pi^{+} q-q)(b^{-}, \cdot)\|_{(0, T)}^{2},
\end{equation}
with $c_{11}(b)=\max \{c / 2, \max \{1, p_{M}\} d / h_{M}\}.$
With the aid of \eqref{globalbound}, the a priori estimate \eqref{estimate-of-enorm} can be derived from the  approximation results of $\pi^{\pm}$ stated in  \cite[Lemmas 3.1-3.2]{Castillo2002MC}.

As pointed out in \cite{Castillo2002MC},  the estimate \eqref{estimate-of-enorm} is optimal in both $h$ and $p$ for smooth solutions, but it is suboptimal in $p$ for singular solutions.  One prototypical example examined in \cite{Castillo2002MC}  tested on the scheme is the exact solution $u(x, t)=x^\pi t.$
The main   theoretical predictions and numerical observations therein are summarized as follows.
\begin{itemize}
\item[(i)] For the $h$-version (with fixed $p$),  the expected  order in $h$ predicted by Theorem \ref{Castillo2002} is $\min \{\pi+0.5, p+1\}$ and $\min \{\pi-0.5, p+1\},$ respectively, for $d=0$ and $d\not=0,$ which is optimal and confirmed by various numerical evidences.
\medskip

\item[(ii)]   For the $p$-version (with fixed $h$),  Theorem \ref{Castillo2002} only leads to the convergence order $2\pi$ in the hyperbolic case (i.e., $d=0$) and $2\pi-2$ in the
convection-diffusion case (i.e., $d\not=0$), but the optimal order  is expected to be  $2\pi+1$ (for $d=0$) and
$2\pi-1$ (for $d\not=0$), respectively (also see \cite[Figure 3]{Castillo2002MC} for numerical justifications). Thus, a loss of one order in $p$ was claimed for the estimates in Theorem \ref{Castillo2002} in both cases: $d=0$ and $d\not=0$.
\end{itemize}

In contrast to the claim (ii) for the case $d\not=0$, we conduct the same numerical tests  as in \cite[Figure~3]{Castillo2002MC} and find that the loss of order in $p$ should be $1/2$ instead. In other words,  the optimal convergence order is $2\pi-\frac 3 2$ rather than $2\pi-1 $ for $d\not=0.$  For comparison, we snapshot  the  \cite[Figure~3]{Castillo2002MC} in Figure~\ref{Figs_1} (left) (where the slopes the error curves were not depicted), and plot the  recovered   results in Figure~\ref{Figs_1} (right) with slopes, which clearly indicate the optimal orders:  $2\pi+1$ for $d=0$ and $2\pi-\tfrac 32$ for $d\not=0.$
We remark that we have adopted the same setup as in \cite{Castillo2002MC} (a fixed uniform mesh with four elements while increasing $p$), so in all experiments we employ a TVD Runge--Kutta time-stepping method \cite{Shu1988JCPEfficient} with sufficiently small time steps so that the overall numerical error is dominated by the spatial error.

% Figure~\ref{Figs_1} presents the $p$--version convergence for the prototypical
% singular solution $u(x,t)=x^\pi t$, which was also used in
% \cite{Castillo2002MC} to illustrate the $p$--suboptimality of the classical
% estimate \eqref{estimate-of-enorm} for nonsmooth solutions.
% For the $p$--version (with fixed $h$), Theorem~\ref{MainEstiA} predicts the
% convergence order $2\pi+1$ for the hyperbolic case ($d=0$) and
% $2\pi-\frac{3}{2}$ for the convection--diffusion case ($d=0.1$), which is
% clearly confirmed by the numerical results.

% For comparison, the right figure is taken from \cite[Figure~3]{Castillo2002MC}.
% For the $p$--version (with fixed $h$), Theorem~\ref{Castillo2002} only yields
% the convergence order $2\pi$ in the hyperbolic case (i.e., $d=0$) and $2\pi-2$
% in the convection--diffusion case (i.e., $d\neq0$).
% The comparison in Figure~\ref{Figs_1} indicates that the present theory captures
% the $p$--version convergence for singular solutions more accurately than the
% existing analysis.

\begin{figure}[!ht]
	\begin{center}
		{~}\hspace*{-20pt}	 \includegraphics[width=0.43\textwidth]{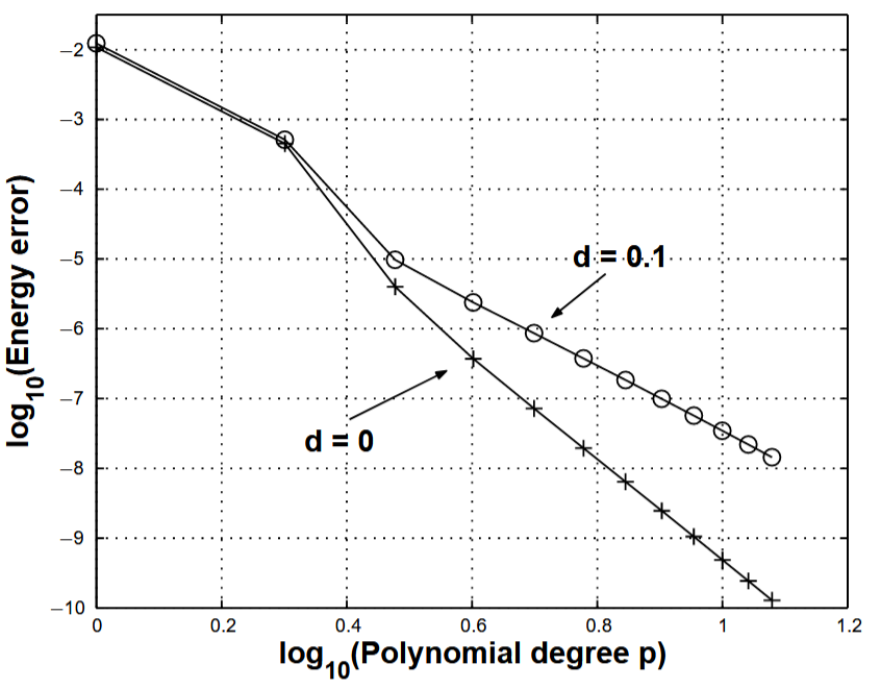}\qquad
\includegraphics[width=0.45\textwidth]{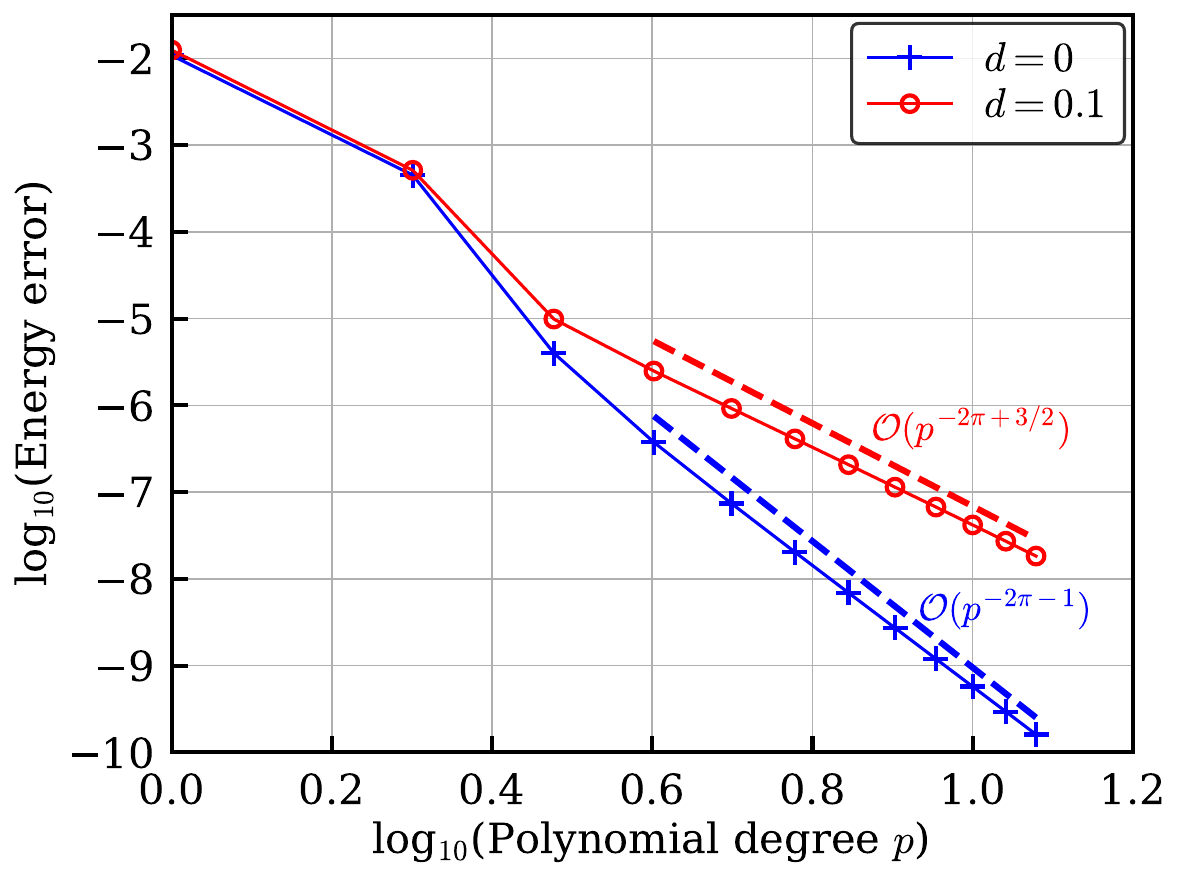}
\caption{Error plots of the $p$-version LDG for the singular solution $u(x,t)=x^\pi t$ with the convection coefficient: $c=0.1$ and the diffusion coefficients: $d=0.1$ and $d=0$.
Left: Snapshot of \cite[Figure~3]{Castillo2002MC}.
Right: Recovery of the same results with slopes of the error plots.}
\label{Figs_1}
	\end{center}
\end{figure}

The remainder of the paper is to close the gap. We shall
develop a new framework to  derive the optimal estimates for singular solutions and explore its extensions.

\section{Optimal error estimates for singular solutions}\label{Sect3:Main-Result}
\setcounter{equation}{0}
\setcounter{lmm}{0}
\setcounter{thm}{0}

In this section, we focus on the analysis of endpoint singularities. The key is to provide an optimal characterization of the {\em fractional regularity} with a delicate analysis of the coefficients of Legendre polynomial expansions, following the spirit of \cite{Liu19MC-Optimal}.

\subsection{Motivating examples}
\label{SecMotivat}
To illustrate the idea, we consider a more general
singular solution
$u(x,t)=x^\alpha (c_1+c_2 x^\beta+\cdots),~
x\in \hat I:=(0,1),~0<t\le T,$
where $\alpha>0$ is non-integer, $\beta> 0$, and $c_1, c_2$ are smooth functions of $t.$ Thus, $u(x,t)$ exhibits a leading algebraic singularity of order $\alpha$, which particularly includes
$
u(x,t) =\mu (x,t)\, x^\alpha,
$
as a special case, where $\mu(x,t)$ is smooth (so it can be expanded in Taylor expansion at $x=0$).
Let $\alpha\in (k-1,k)$ with integer $k\ge 1.$
% where $k:=\lceil\alpha\rceil$ and $\{\alpha\}\in (0,1)$ is the fractional part of $\alpha.$
Then,
direct calculation leads to
$$
\partial_x^i u=\Big(c_1\frac{\Gamma(\alpha+1)}{\Gamma(\alpha-i+1)}  +c_2\frac{\Gamma(\alpha+\beta+1)}{\Gamma(\alpha+\beta-i+1)} x^{\beta} +\cdots\Big) x^{\alpha-i}, \quad i=1,\ldots, k,
$$
which implies $u,\partial_x u, \cdots, \partial_x^k u\in  L^1(\hat I).$
%with $I_1=(0,x_1).$
We next show that $u(x,t)$ admits a ``fractional derivative'' of order $k<\alpha+1<k+1,$  which is the key to recovering of the optimal order $2\alpha+1$  (for $d=0$) in $p$ (see Theorem \ref{EndPi}).  Indeed, using the Riemann-Liouville fractional integral formula (see e.g., \cite[p. 81]{2006KilbasBookTheory}):
$${\mathcal I}_{0+}^{s} x^{\nu}=\frac{1}{\Gamma(s)}\int_{0}^{x}
 \frac{y^\nu}{(x-y)^{1-s}}
 \,{\rm d}y=\frac{\Gamma(\nu+1)}{\Gamma(\nu+s+1)} x^{\nu+s},\quad \nu>-1,\;\; s>0, $$
 we immediately derive
\begin{equation*}
\begin{split}
 %\CFrD{\alpha}{0+} u(x,t)&=
 {\mathcal I}_{0+}^{k-\alpha} \partial_x^k u
 %=\big(I_{0+}^{k-\alpha} u^{(k)}\big)(x,t)
 & =c_1\frac{\Gamma(\alpha+1)}{\Gamma(\alpha-k+1)}
 {\mathcal I}_{0+}^{k-\alpha} x^{\alpha-k}\,+
c_2\frac{\Gamma(\alpha+\beta+1)}{\Gamma(\alpha+\beta-k+1)}
 {\mathcal I}_{0+}^{k-\alpha} x^{\alpha+\beta-k}\,+\cdots\\
 &=c_1\frac{\Gamma(\alpha+1)}{\Gamma(\alpha-k+1)}
  \Gamma(\alpha-k+1)+ c_2\frac{\Gamma(\alpha+\beta+1)}{\Gamma(\alpha+\beta-k+1)}
 \frac{\Gamma(\alpha+\beta-k+1)}{\Gamma(\beta+1)} x^{\beta}\, +\cdots
 \\
 &
=c_1\Gamma(\alpha+1)+
c_2\frac{\Gamma(\alpha+\beta+1)}{\Gamma(\beta+1)} x^{\beta} +\cdots= \CFrD{\alpha}{0+} u,
% \, t+\cdots
% \in W^{\lceil \beta \rceil,1}(I_1).
\end{split}
\end{equation*}
 which turns out to be  the order-$\alpha$ (left) Caputo fractional derivative of $u$ as defined in e.g., \cite{2006KilbasBookTheory}. It is clear that we can take at least one more derivative as
 $\partial_x (\CFrD{\alpha}{0+} u)=\mathcal{O}(x^{\beta-1})\, \in L^1(I).$ Interestingly, if we revisit  the example  $u(x, t)=x^\pi t$ in
 \cite{Castillo2002MC}, then  $\CFrD{\pi}{0+} u=\Gamma(\pi+1) t \in C^\infty$.  Indeed, as to be shown in the proof of Theorem~\ref{EndPi} (see \eqref{Unend-0}),
the differentiability property
$\partial_x^m (\CFrD{\alpha}{0+} u)$ and the boundedness of
$|\partial_x^i (\CFrD{\alpha}{0+} u)(0+)|$ for all $0 \le i \le m$
directly determine the order in $p$ for the Gauss--Radau projection errors.
More importantly, the functional spaces and norms that arise are naturally
dictated by the exact formula for the Legendre expansion coefficients,
obtained through integer and fractional integration by parts
(see \eqref{Unend-0}), even though the notation appears a bit involved.

With the above insights, we now introduce the space of functions to optimally characterize the fractional regularity of the singular solutions.  Let $\Lambda:=(\hat a,\hat b)$ be a generic finite interval, which will be taken as $I_1=(a,x_1)$ and $I_{\rm ref}=(-1,1),$ etc. later on. Let $W^{k,1}(\Lambda)\subset L^1(\Lambda)$ for integer $k\ge 1$  be the usual  Sobolev space as in Admas \cite{2003AdamsBookSobolev}.
It is seen from the above example that we need to  characterize the regularity of  both $u$
and its fractional derivative, leading to the space with two regularity indexes below. More precisely, for $\alpha\in (k-1,k)$ and $ k,m\in \mathbb N,$ we define the space
\begin{equation}\label{defn:FracSpa-AB}
{\mathcal U}_{\hat a+}^{\alpha, m}(\Lambda):=\big\{u\,:\, u\in {W}^{k,1}(\Lambda)\;\; {\rm and}\;\;
\CFrD{\alpha}{\hat a+} u \in  {W}^{m,1}(\Lambda) \big\},
\end{equation}
 associated with the semi-norm:
\begin{equation}
\label{defn:FracSpa-3-1}
%{U_{a+}^{\alpha, k}}:=V_{\Omega}\big[v^{(k)}\big]+|\sin (s \pi)| \sum_{i=0}^k\big|v^{(i)}(a+)\big|,
{U_{\hat a+}^{\alpha, m}}({u;\Lambda}):=\big\|\MD^{m}\big(\CFrD{\alpha}{\hat a+}   u\big)\big\|_{{L^1}(\Lambda)}+ \sum_{i=0}^{m-1}\big|\MD^i \big(\CFrD{\alpha}{\hat a+}   u\big) (\hat a+)\big|,
\end{equation}
where ${\mathcal D}^m$ is the
$m$-th integer-order derivative operator. Note that for any  $u\in {\mathcal U}_{\hat a+}^{\alpha, m},$
 \(\MD^i \big(\CFrD{\alpha}{\hat a+}   u\big) (\hat a+)\) for \(i=0,1,\dots,m-1\) are  finite, in view of
 the Sobolev embedding property (see \cite[P.~212]{2011BrezisBookFunctional}).

 \subsection{Main result}
 Now, we are ready to present our optimal convergence result for the class of solutions with (left) endpoint singularities.
 Define the
   broken Sobolev space associated to the mesh $\mathcal T:$
\begin{equation}\label{SpaceL0}
{\mathcal W}^{\alpha, m, s}_{a+}(I;\mathcal{T}):=\big\{v: I=(a,b)\to \mathbb{R}\,:\, v|_{I_{1}} \in  {\mathcal U_{a+}^{\alpha, m}}(I_1),\;  v|_{I_j}\in H^{s+1}(I_{j}),\; j=2, \ldots, M\big\},
\end{equation}
equipped with the semi-norm:
\begin{equation}\label{NormL0}
		\begin{split}
 \|u\|_{\mathcal W}:=\|u\|_{{\mathcal W}^{\alpha, m, s}_{a+}(I;\mathcal{T})}:= {U_{a+}^{\alpha,m}}
 (u;{I_1})
  + \sum_{j=2 }^M
 \|u^{(s+1)}\|_{L^2(I_{j})}.
\end{split}
		\end{equation}

\begin{thm}\label{MainEstiA}
Let $u$ be the solution to \eqref{CDEq} and
$q=\sqrt{d}\,u_x,$ and let $(u_{h}^{p},q_{h}^{p})^{\T}$ be the LDG solution to
\eqref{LDG-1}–\eqref{ubsUB}.
% \begin{comment}
% the following regularity holds:
% \begin{itemize}
%  \item[(i)] For the first interval $I_{1}=(x_{0}, x_{1})$, there exists $\alpha>0$ such that
%  $u\in {\mathcal U}^{\alpha, m}_{x_{0}+}(I_{1});$
%  \item[(ii)] For the other intervals $I_{l}=(x_{l-1}, x_{l}), \ l=2,3,\ldots,M$,
% there exists an integer $s_l\ge0$ such that
%  $u\in H^{s_l+1}(I_l). $
%  \end{itemize}
% \end{comment}
%  \begin{comment}
% We define the corresponding norm
% \begin{equation*}
% 		\begin{split}
%  \VertV{u}_{A}:= {U_{x_{0}+}^{\alpha, m}}
%  (u, {I_1})
%   + \sum_{l=2 }^M
%  \|u^{( {s}_l+1)}\|_{L^2(I_{l})}.
% \end{split}
% 		\end{equation*}
%   \end{comment}
Denote and assume
\begin{equation}\label{NewNormA}
\begin{split}
 \VertV{u}_{\mathcal{E}, T}:=
2\sup _{0\le t\le T}\|u(\cdot, t)\|_{\mathcal W}
+
3\sqrt{d}\,\Big(\int_{0}^{T}
\|u(\cdot, t)\|_{\mathcal W} ^2 \, \rd t\Big)^{1/2}
+
\int_{0}^{T} \|u_t(\cdot, t)\|_{\mathcal W}\,\rd t<\infty.
\end{split}\end{equation}
%Assume that for each fixed $t\in[0,T],$ $u(\cdot,t), u_t(\cdot,t)\in {\mathcal W}^{\alpha, m, s}_{a+}(I;\mathcal{T}).$
% Assume that $\VertV{u}_{\mathcal{E}, T}<\infty$.
Then for  $h=\max\{h_j\}$ and  $p=\min\{p_j\}$ with  $p>\alpha+m-1,$  we have
\begin{itemize}
\item[(i)]    for $d=0$ and $p\ge s,$
\begin{equation}\label{MainII-a-d0}
\|(u-u_{h}^{p})(\cdot,T)\|
\le
C\Big(
\frac{h^{\alpha+1/2}}{p^{\min\{2\alpha+1,\alpha+m-1/2\}}}
+\frac{h^{s+1}}{p^{s+1}}
\Big)
\VertV{u}_{\mathcal{E},T};
\end{equation}
\item [(ii)]  for $d\neq0$ and $p\ge s-1,$
\begin{equation}\label{MainII-a-dp}
\|(u-u_{h}^{p})(\cdot,T)\|
+\|q-q_{h}^{p}\|_{Q_T}
\le
C\Big(
\frac{h^{\alpha-1/2}}{p^{\min\{2\alpha-3/2,\alpha+m-3/2\}}}
+\frac{h^{s}}{p^{s}}
\Big)
\VertV{u}_{\mathcal{E},T},
\end{equation}
\end{itemize}
% \begin{equation}\label{MainII-a}
% \begin{split}
% \|(u-u_N)(\cdot,T)\|+\|q-q_N\|_{Q_{T}} \leq C \Big(
% \frac{h^{\alpha+1/2}}{p^{\min\{2\alpha+1,\alpha+m-1/2\}}}
%   +  \frac{h^{\min \{s, p\}+1}}{p^{s+1}}\Big)\Big(1+
%   \sqrt{d}\, \frac{p^2}{h}\Big)
%  \VertV{u}_{\mathcal{E}, T}^{A},
% \end{split}
% \end{equation}
where the constant
$C$ depends on $\alpha, s,$  but is independent of $h,$ $p$ and $u$.
\end{thm}
% \comm{\color{red} Need to fix Remark 3.1!}
\begin{rem}\label{sing-rmk} We revisit the motivating example:
$u(x,t)=x^\alpha (c_1+c_2 x^\beta+\cdots),~
x\in  I=(a,b)=(0,1),$
where $\alpha>0$ is non-integer, $\beta> 0$, and $c_1, c_2$ are smooth functions of $t\in [0,T],$ studied  in Subsection~\ref{SecMotivat}.  One verifies from the definition  \eqref{defn:FracSpa-AB}-\eqref{defn:FracSpa-3-1} and  previous analysis that for fixed $t>0,$ the leading term $c_1  x^\alpha \in {\mathcal U}_{0+}^{\alpha,\infty}(I_1)$ and the second term $c_2  x^{\alpha+\beta} \in {\mathcal U}_{0+}^{\alpha+\beta,\infty}(I_1),$ so the corresponding orders in $p$ are  ${\mathcal O}(p^{-(2\alpha+1)})$ (for $d=0$) and  ${\mathcal O}(p^{-(2\alpha-3/2)})$ (for $d\not=0$), respectively.
Indeed, as shown in Figure~\ref{Figs_1}, the theoretical predictions and numerical errors agree well.

It is seen from   \eqref{MainII-a-d0} (resp.~\eqref{MainII-a-dp}) that if $m+\alpha-1/2\le 2\alpha+1,$ i.e., $m\le \alpha+3/2$ (resp.~$m\le \alpha$)  then
the regularity of $\CFrD{\alpha}{\hat a+} u$ dominates, and the convergence order is ${\mathcal O}(p^{-(\alpha+m-1/2)})$ for $d=0$ (resp.~${\mathcal O}(p^{-(\alpha+m-3/2)})$ for $d\not=0$).
In  Subsection~\ref{Subsec43}, we will provide some numerical illustrations of such a convergence behavior, see
Figure~\ref{Figs4-2} (right).
\qed
\end{rem}

To prove the main results stated in the above theorem,
we first make necessary preparations.
\subsection{Optimal estimates for Gauss-Radau projections}
For any $\omega\in L^1(I_{\rm ref})$ with
$I_{\rm ref}=(-1,1),$ its Legendre expansion coefficient and the corresponding $L^2$-orthogonal projection are given by
\begin{equation}\label{ancoef}
\hat {\omega}_n=\frac{2n+1}{2}\int_{-1}^{1}{\omega(\xi )L_n(\xi )}\, {\rm d} \xi , \quad
 (\Proj_p\, \omega)(\xi )= \sum_{n=0}^p \hat {\omega}_n\, L_n(\xi ),
\end{equation}
where $L_n(\xi )$ is the Legendre polynomial of degree $n.$
% We recall the $L^2$-orthogonal projection:
% \begin{equation}\label{Cbexp}
% \omega(\xi )=\sum_{n=0}^{\infty}\hat {\omega}_n\, L_n(\xi ),\quad\;\;   (\Proj_p\, \omega)(\xi )= \sum_{n=0}^p \hat {\omega}_n\, L_n(\xi ).
% \end{equation}
As in \eqref{7projections}, we define the Gauss-Radau projection on the reference interval: $\Proj_p^\pm: W^{1,1}(I_{\rm ref})\to  \mathcal{P}_p(\RefDomain)$ such that
% find $\Proj_p^\pm \omega \in \mathcal{P}_p(\RefDomain)$
% with $\Proj_p^\pm \omega(\mp 1) = \omega(\mp 1)$ such that
% find $\Proj^\pm_p \omega\in \mathcal{P}^p(\RefDomain)$ and $ \Proj^\pm_p \omega(\mp1)=\omega(\mp1)$ such that
\begin{equation}\label{LPiMinus-1}
 (\Proj^{\pm}_{p} \omega- \omega, v)_{\RefDomain}=0,\quad  \forall\, v \in \mathcal{P}_{p-1}(\RefDomain);\quad  \Proj^\pm_p \omega(\mp1)=\omega(\mp1).
% \int^1_{-1}(\Proj^\pm_p \omega-\omega)\,v\, {\rm d} \xi =0, \quad \forall v\in \mathcal{P}^{p-1}(\RefDomain).
	\end{equation}	

We shall use the following fundamental identities (see
 \cite[Proposition 3.12]{Castillo2002MC}).
\begin{lmm}\label{LPiMinus}
For any $\omega\in W^{1,1}(I_{\rm ref}),$ we have
\begin{equation}\label{LPiMinus-2}
		\begin{split}
		\|\omega-\Proj_p^{\pm}  \omega\|_{L^2(I_{\rm ref})}^2= \|\omega-\Proj_p  \omega\|_{L^2(
        I_{\rm ref})}^2+\frac{2}{2p+1}\Big(\frac{\hat \omega'_{p}}{2p+1}\mp \frac{\hat \omega'_{p+1}}{2p+3}\Big)^2,
		\end{split}
		\end{equation}	
and
\begin{equation}\label{LPiMinus-3}
		\begin{split}
		|(\omega-\Proj_p^{\pm}  \omega)({\pm}1)|=\frac{2}{2p+1}|\hat \omega'_{p}|, \quad \hat \omega'_{p}=\frac{2p+1}{2}\int_{-1}^{1}{\omega'(\xi )L_p(\xi )}\, {\rm d} \xi .
		\end{split}
		\end{equation}
% which for $\Proj^+_p$:
% \begin{equation}\label{LPiMinus-4}
% 		\begin{split}
% 		\|\omega-\Proj_p^+  \omega\|_{L^2(I_{\rm ref})}^2= \|\omega-\Proj_p  \omega\|_{L^2(
%         I_{\rm ref})}^2+\frac{2}{2p+1}\Big(\frac{\hat \omega'_{p}}{2p+1}-\frac{\hat \omega'_{p+1}}{2p+3}\Big)^2,
% 		\end{split}
% 		\end{equation}
% and
% \begin{equation}\label{LPiMinus-5}
% 		\begin{split}
% 		|(\omega-\Proj_p^+  \omega)(1)|=\frac{2}{2p+1}|\hat \omega'_{p}|.
% 		\end{split}
% 		\end{equation}
\end{lmm}

The above lemma implies that we can estimate the $L^2$-projection and the last two expansion coefficients of  $\omega'(\xi)$ to derive the Gauss-Radau projection errors as follows.
\begin{thm}\label{EndPi}
Let $\alpha \in (k-1, k)$ and $k,m\in \mathbb N.$ Then for any
 $\omega\in {\mathcal U}^{\alpha,m}_{-1+}(\RefDomain)$ with $1/2< \alpha+m<  p+1,$ we have
% Assume the conditions in Lemma {\rm \ref{lmm:35}} hold,then for $0< \alpha+m< p+1,$
\begin{equation}\label{EndPi-1-1}
		\begin{split}
		\|\omega-\Proj_p^-  \omega\|_{L^2(\RefDomain)}\le C p^{-\min\{2\alpha+1,\alpha+m-1/2\}}\rFraSLNorm,
		\end{split}
		\end{equation}
        \begin{equation}\label{EndPiII-1-1}
		\begin{split}
		\|\omega-\Proj_p^+ \omega\|_{L^2(\RefDomain)}\le  C p^{-\min\{2\alpha+1/2,\alpha+m-1/2\}}\rFraSLNorm,
		\end{split}
		\end{equation}
and
\begin{equation}\label{EndPi-2-1}
		\begin{split}
		|\omega(\pm1)-\Proj_p^\pm  \omega(\pm1)|\le Cp^{-\min\{2\alpha,\alpha+m-1/2\}}\rFraSLNorm,
		\end{split}
		\end{equation}
where $C$ is a positive constant independent of $p$ and $\omega$.
%For $\alpha=k,$
% the estimates  \eqref{FracL2B} hold,  with the total variation $V_{\bar\Omega}[\omega^{(k)}]$ is in place of $\rFraSNorm .$
\end{thm}

\begin{proof}
The proof essentially relies on a delicate analysis of  the exact formula of the  Legendre expansion coefficient derived from integer and fractional integration by parts (see \cite[Theorem 4.1]{Liu2021ACMOptimal}):
  for $n>\alpha+m,$
\begin{equation}\label{Unend-0}
		\begin{split}
		\hat {\omega}_n & =\frac{2n+1}{2^{\alpha+m+1}\Gamma(\alpha+m+1)}\int_{-1}^{1}
        (1-\xi ^2)^{\alpha+m}\,
{}^{r\!} G_{n-\alpha-m}^{(\alpha+m+1/2)}(\xi )\, v^{(m)}(\xi ) \, {\rm d}\xi \\
&\quad -  \sin(\alpha\pi)\sum_{i=0}^{m-1} (-1)^{n+i}\,  \widehat{C}_{n,\alpha+i}\, v^{(i)} (-1+),
		\end{split}
		\end{equation}
where
\begin{equation}\label{newconstants-1}
\begin{split}
 v(\xi):=\CFrD{\alpha}{-1+}   \omega(\xi ),\quad \widehat{C}_{n,\beta}:=  2^{\beta}\Gamma(\beta+1)  \frac{2n+1}{\pi}  \frac{\Gamma(n-\beta)}{\Gamma(n+\beta+2)},
\end{split}\end{equation}
and ${}^{r\!} G_{n-\alpha-m}^{(\alpha+m+1/2)}(\xi )$ is the
generalized Gegenbauer function of fractional degree (GGF-F),  which is defined via the hypergeometric function (cf. \cite[Def.~2.1]{Liu19MC-Optimal}):
for real $\lambda>-1/2$ and real $\nu \ge 0,$
\begin{equation*}\label{rgjfdef}
{}^{r\!}G_\nu^{(\lambda)}(\xi ):=
{ }_2 F_1\Big[\begin{array}{c}
-\nu, \nu+2\lambda \\\lambda+\frac{1}{2}\end{array} ; \frac{1-\xi } 2\Big],\quad \xi\in (-1,1).
\end{equation*}
%where ${ }_2 F_1$ is the hypergeometric function.
It is clear that \eqref{Unend-0} holds provided that $v^{(m)} \in L^{1}(I_{\rm ref})$ and the summation term is finite. This naturally motivates the introduction of the refined fractional space ${\mathcal U}^{\alpha,m}_{-1+}(\RefDomain)$ in \eqref{defn:FracSpa-AB}, which enhances the framework developed in \cite{Liu19MC-Optimal,Liu2021ACMOptimal}.
Indeed, we retain the term with the factor $\sin (\alpha \pi)$ to include the integer case as a special case. Moreover, as $\alpha \rightarrow 0$, the following new bound becomes  sharper.
% Naturally, the  new fractional space  ${\mathcal U}^{\alpha,m}_{-1+}(\RefDomain)$ is obtained from \eqref{Unend-0}.
% We refine  the new bound  for $\hat {\omega}_n$ in new fractional space,
% we keep $\sin(\alpha\pi)$ to
% contain integer case, if $\alpha \to 0$, the new following bound becomes much sharper.
More precisely, using   \cite[(4.30)]{Liu19MC-Optimal} to estimate the involved GGF-F, we get
% Using Lemma \ref{LBoundForGegPoly}, we have
\begin{equation}\label{BoundForUnEnd}
\begin{split}
|\hat {\omega}_n|& \le
\bigg\{
	\frac{ (2n+1)\Gamma( ({n-\alpha-m+1})/ 2)}{2^{\alpha+m+1} \sqrt{\pi}\Gamma( ({n+\alpha+m})/2+1)}
+ |\sin(\alpha\pi)|\sum_{i=0}^{m-1}\widehat{C}_{n,\alpha+i}\bigg\}\,  {U_{-1+}^{\alpha, m}}({\omega;\RefDomain}).
	\end{split}
	\end{equation}
%    Let us deal with
%   $$ \sum_{i=0}^{m-1}\widehat{C}_{n,\alpha+i}$$
In view of \eqref{newconstants-1}, we deduce from direct calculation that
\begin{equation}\label{BoundForUnEnd-1}
\begin{split}
	\sum_{i=0}^{m-1}\frac{\widehat  C_{n,\alpha+i}}{\widehat  C_{n,\alpha}}&=1+ 	\sum_{i=1}^{m-1}
    \bigg( \prod_{r=1}^i\frac{2(\alpha+r)}{(n-\alpha-r)(n+\alpha+r+1)}\bigg)\\
    &
    \le  \sum_{i=0}^{m-1}\Big(\frac{2(\alpha+m-1)}{n+\alpha+m}\Big)^i <  \sum_{i=0}^{\infty}\Big(\frac{2(\alpha+m-1)}{n+\alpha+m}\Big)^i=\frac{n+\alpha+m}{n-\alpha-m+2},
    \end{split}
	\end{equation}
    where we used the fact:  for  $1\le r\le i\le m-1$ and $ n > \alpha+m, $
	\begin{equation*}
(n-\alpha-r)(n+\alpha+r+1)\ge (n-\alpha-m+1)(n+\alpha+m)\ge n +\alpha + m.
	\end{equation*}
Consequently, we obtain from \eqref{BoundForUnEnd} and \eqref{BoundForUnEnd-1} that
\begin{equation}\label{BoundUnEnd-ReNew}
	\begin{split}
|\hat {\omega}_n|& \le
\bigg\{
	\frac{ (2n+1)\Gamma( ({n-\alpha-m+1})/ 2)}{2^{\alpha+m+1} \sqrt{\pi}\Gamma( ({n+\alpha+m})/2+1)} + |\sin(\alpha\pi)| \frac{n+\alpha+m}{n-\alpha-m+2}\widehat{C}_{n,\alpha}\bigg\}\,   \rFraSLNorm.
	\end{split}
	\end{equation}
Thus, using an argument similar to \cite[(E.16)]{Liu2021ACMOptimal}, we can obtain the  refined $L^2$-bound:
\begin{equation}\label{FracEndL2B}
		\begin{split}
		\|\omega&-\Proj_p \omega\|_{L^2(\RefDomain)}\le\bigg\{ \frac{4}{(2\alpha+2m-1) \pi}
		\frac{ \Gamma( {p-\alpha-m+2})}{\Gamma( {p+\alpha+m+1})}+ \frac{(2p+3)\,\Gamma^2( ({p-\alpha-m})/ 2+1)}{ 2^{2(\alpha+m)}{\pi}\,\Gamma^2( ({p+\alpha+m+3})/2) }\\
& + |\sin(\alpha\pi)|\frac{2^{2(3\alpha+2)}\Gamma^2(\alpha+1)}{\pi^2(2\alpha+1)} \frac{(p+\alpha+m+1)^2}{({p-\alpha-m+3})^2}
\frac{p+3/2}{p+1}
\frac{\Gamma(2p-2\alpha+1)}{\Gamma(2p+2\alpha+3)}
\bigg\}^{1/2}\, {U_{-1+}^{\alpha, m}}({\omega;\RefDomain}).
		\end{split}
		\end{equation}

 We next estimate the expansion coefficients of $\omega'(\xi)$ in  Lemma \ref{LPiMinus}. From \eqref{Unend-0}, we have
\begin{equation}\label{dUnend-0}
		\begin{split}
		\hat {\omega}'_n =&\frac{2n+1}{2^{\alpha+m}\Gamma(\alpha+m)}\int_{-1}^{1}  (1-\xi ^2)^{\alpha+m-1}\,
{}^{r\!} G_{n-\alpha-m+1}^{(\alpha+m-1/2)}(\xi ) \, v^{(m)}(\xi ) \, {\rm d}\xi \\
        &+  \sin(\alpha\pi)\sum_{i=0}^{m-1} (-1)^{n+i}  \widehat{C}_{n,\alpha+i-1}\, v^{(i)} (-1+).
%\big\{{\mathcal D}^j\big( {}_{-1}I_{\xi }^{1-s}\omega^{(m)}  (\xi )\big) \big\}
		\end{split}
		\end{equation}
Using \eqref{dUnend-0}, taking the maximum value of
$ (1-\xi ^2)^{\alpha+m-1}\,
{}^{r\!} G_{n-\alpha-m+1}^{(\alpha+m-1/2)}(\xi ),$
and
proceeding similarly to the proof of \eqref{BoundUnEnd-ReNew}, we obtain
\begin{equation}\label{EndPi-3}
\begin{split}
|\hat \omega'_n| &\le \bigg\{\frac{1}{2^{\alpha+m} \sqrt{\pi}}
	\frac{ (2n+1)\Gamma( ({n-\alpha-m})/ 2+1)}{\Gamma( ({n+\alpha+m+1})/2)}\\
    &
    \quad
+ |\sin(\alpha\pi)| \frac{n+\alpha+m-1}{n-\alpha-m+3}\widehat{C}_{n,\alpha-1}\bigg\}\,   \rFraSLNorm .
\end{split}
\end{equation}
Taking $n=p$ in \eqref{EndPi-3} and using  \eqref{LPiMinus-3}, we obtain \begin{equation}\label{EndPi-2}
		\begin{split}
		|\omega(-1)-\Proj_p^-  \omega(-1)|\le \widehat{\Phi}^{(\alpha,m)}_p\,\rFraSLNorm ,
		\end{split}
		\end{equation}
        where we denote
$$\widehat{\Phi}^{(\alpha,m)}_p:= \frac{1}{2^{\alpha+m-1} \sqrt{\pi}}
	\frac{ \Gamma( ({p-\alpha-m})/ 2+1)}{\Gamma( ({p+\alpha+m+1})/2)}		+|\sin(\alpha\pi)|\frac{2^{\alpha}\Gamma(\alpha) }{\pi}
        \frac{p+\alpha+m-1}{p-\alpha-m+3}
\frac{\Gamma(p-\alpha+1)}{\Gamma(p+\alpha+1)}.  $$
% The above new bound also includes the integer cases, either as $\alpha \to 0$ or as $p\to \alpha+m+1.$
By  \eqref{dUnend-0},
% \begin{equation}\label{EndPi-4}
%  \begin{split}
% \frac{\hat \omega'_{p}}{2p+1} & =
%  \frac{1}{2^{\alpha+m} \Gamma(\alpha+m)}
% \int_{-1}^{1}
%   {}^r \mathcal{G}_{p}^{(\alpha+m-1)}(\xi )\, v^{(m)}(\xi ) \, {\rm d}\xi
% \\& \quad
% +\sin(\alpha\pi)\sum_{i=0}^{m-1}(-1)^{p+i}
%  \frac{2^{\alpha+i-1}\Gamma(\alpha+i)}{\pi}
%    \frac{\Gamma(p-\alpha-i+1)}{\Gamma(p+\alpha+i+1)}
% v^{(j)}(-1+).
%  \end{split}
% 	\end{equation}
% Thus
 \begin{equation}\label{EndPi-5}
 \begin{split}
\frac{\hat \omega'_{p}}{2p+1}+\frac{\hat \omega'_{p+1}}{2p+3} & =
 \frac{1}{2^{\alpha+m} \Gamma(\alpha+m)}
\int_{-1}^{1}
(1-\xi ^2)^{\alpha+m-1}
  \big({}^{r\!} G_{p-\alpha-m+1}^{(\alpha+m-1/2)}(\xi ) \\
  &+ {}^{r\!} G_{p-\alpha-m+2}^{(\alpha+m-1/2)}(\xi ) \big)\,v^{(m)}(\xi ) \, {\rm d}\xi
+\sin(\alpha\pi)\sum_{i=0}^{m-1}(-1)^{p+i}
\widetilde  C_{p,\alpha+i}
v^{(i)}(-1+),
 \end{split}
	\end{equation}
    where
    $$
\widetilde  C_{p,\beta} :
=\frac{2^{\beta }\Gamma(\beta+1)}{\pi}\frac{\Gamma( p-\beta+1)}{\Gamma(p+\beta+2)}.
$$
%where
%$$\widetilde  C_{p,\beta} := \frac{2^{\beta-2}\Gamma(\beta-1)}{\pi} \Big(\frac{\Gamma(p-\beta+2)}{\Gamma(p+\beta)}-   \frac{\Gamma(p-\beta+3)}{\Gamma(p+\beta+1)}\Big)
%=\frac{2^{\beta-1}\Gamma(\beta)}{\pi}\frac{\Gamma( p-\beta+2)}{\Gamma(p+\beta+1)}.
%$$
Similarly,  using \cite[(4.30)]{Liu19MC-Optimal} yields
%$$|{}^r \mathcal{G}_{n}^{(\beta)}(\xi )|\le \frac{\Gamma(\beta+1)\Gamma((n-\beta+1)/2)}{\sqrt \pi \Gamma((n+\beta)/2+1)}$$
\begin{equation}\label{EndPi-6}	
  \frac{(1-\xi ^2)^{\alpha+m-1}}{2^{\alpha+m} \Gamma(\alpha+m)}\big|
{}^{r\!} G_{p-\alpha-m+1}^{(\alpha+m-1/2)}(\xi )
  + {}^{r\!} G_{p-\alpha-m+2}^{(\alpha+m-1/2)}(\xi ) \big| \le
\frac{1}{2^{\alpha+m-1} \sqrt{\pi}}
	\frac{ \Gamma( ({p-\alpha-m})/ 2+1)}{\Gamma( ({p+\alpha+m+1})/2)}.
\end{equation}
Similar to the derivation of
\eqref{BoundForUnEnd-1}, we can obtain
\begin{equation}\label{EndPi-7}
	\sum_{i=0}^{m-1}
\frac{\widetilde  C_{p,\alpha+i}}
{\widetilde  C_{p,\alpha}}
=1+	\sum_{i=1}^{m-1}
    \bigg( \prod_{r=1}^i\frac{2(\alpha+r)}{(p-\alpha-r+1)(p+\alpha+r+1)}\bigg)
<\frac{p+\alpha+m}{p+1}.
	\end{equation}
% From straightforward calculations, we find
% \begin{equation}\label{EndPi-7}
% 	\begin{split}
% 	\sum_{j=0}^{m}
% \frac{\widetilde  C_{p,\alpha+j}}{\widetilde  C_{p,\alpha}}&=
% \sum_{j=0}^{m}2^j
% \frac{\Gamma(\alpha+j)}{\Gamma(\alpha)}
% \frac{\Gamma(p-\alpha-j+2)}{\Gamma(p-\alpha+2)}
% \frac{\Gamma(p+\alpha+1)}{\Gamma(p+\alpha+j+1)}
% \\&
% =1+\sum_{j=1}^{m}\bigg( \prod_{i=1}^j
% \frac{2(\mu+i-1)}{(p-\alpha-i+2)(p+\alpha+i)}\bigg)
% \\&
% \le 1+\sum_{j=1}^{m}
% \Big(\frac{\alpha+m-1}{p+1/2}\Big)^j
% <  \sum_{j=0}^{\infty}\Big(\frac{\alpha+m-1}{p+1/2}\Big)^j
% \\&=\frac{p+1/2}{p-\alpha-m+3/2},
% 	\end{split}
% 	\end{equation}
% where we used the fact:  for  $1\le i\le j\le m$ and $  \alpha+m<p+1, $
% 	\begin{equation*}
% (p-\alpha-i+2)(p+\alpha+i)\ge (p-\alpha-m+2)(p+\alpha+m)\ge 2p+1.
% 	\end{equation*}
Thus, by \eqref{EndPi-5}-\eqref{EndPi-7},
\begin{equation}\label{EndPi-8}
		\begin{split}	
\Big(\frac{\hat \omega'_{p}}{2p+1}+\frac{\hat \omega'_{p+1}}{2p+3}\Big)^2 & \le \bigg\{\frac{1}{2^{\alpha+m-1} \sqrt{\pi}}
	\frac{ \Gamma( ({p-\alpha-m})/ 2+1)}{\Gamma( ({p+\alpha+m+1})/2)}
		\\
&\quad +|\sin(\alpha\pi)|\frac{p+\alpha+m}{p+1}\widetilde  C_{p,\alpha}
		\bigg\}^2\big(\rFraSLNorm\big)^2.
		\end{split}
		\end{equation}
From \eqref{LPiMinus-2}, \eqref{FracEndL2B} and \eqref{EndPi-8}, we have
\begin{equation}\label{EndPi-1}
		\begin{split}
		\|\omega-\Proj_p^-  \omega\|_{L^2(\RefDomain)}\le&\widehat{\Upsilon}^{(\alpha,m)}_p\, \rFraSLNorm ,
		\end{split}
		\end{equation}	
where we denote
%\begin{equation*}
%\begin{split}
%\widehat{\Upsilon}^{(\alpha,m)}_p:=&\bigg\{ \frac{4}{(2\alpha+2m+1) \pi}
%		\frac{ \Gamma( {p-\alpha-m+1})}{\Gamma( {p+\alpha+m+2})}+ \frac{(p+1/2)\,\Gamma^2( ({p-\alpha-m+1})/ 2)}{ 2^{2\alpha+2m+1}{\pi}\,\Gamma^2( ({p+\alpha+m})/2+2) }\\
%&\qquad + \frac{2^{6\alpha+5}\Gamma^2(\alpha+1)}{\pi^2(2\alpha+1)} \frac{(p+1)^2\,\Gamma(2p-2\alpha+1)}{({p-\alpha-m+1})^2\,\Gamma(2p+2\alpha+3)}\\
%&\qquad +\frac{2}{2p+1}\bigg\{\frac{1}{2^{\alpha+m} \sqrt{\pi}}
%	\frac{ \Gamma( ({p-\alpha-m+1})/ 2)}{\Gamma( ({p+\alpha+m+2})/2)}
%		+\frac{2^{\alpha}\Gamma(\alpha) p\Gamma(p-\alpha+1)}{\pi\,(p-\alpha-m+1) \Gamma(p+\alpha+1)}
%		\bigg\}^2
%\bigg\}^{1/2},
%\end{split}
%	\end{equation*}
\begin{equation*}
\begin{split}
\widehat{\Upsilon}^{(\alpha,m)}_p&:=\bigg\{ \frac{4}{(2\alpha+2m-1) \pi}
		\frac{ \Gamma( {p-\alpha-m+2})}{\Gamma( {p+\alpha+m+1})}+ \frac{(2p+3)\,\Gamma^2( ({p-\alpha-m})/ 2+1)}{ 2^{2(\alpha+m)}{\pi}\,\Gamma^2( ({p+\alpha+m+3})/2) }\\
&\quad + |\sin(\alpha\pi)|\frac{2^{2(3\alpha+2)}\Gamma^2(\alpha+1)}{\pi^2(2\alpha+1)} \frac{(p+\alpha+m+1)^2}{({p-\alpha-m+3})^2}
\frac{p+3/2}{p+1}
\frac{\Gamma(2p-2\alpha+1)}{\Gamma(2p+2\alpha+3)}\\
&\quad +\frac{2}{2p+1}\bigg\{\frac{1}{2^{\alpha+m-1} \sqrt{\pi}}
	\frac{ \Gamma( ({p-\alpha-m})/ 2+1)}{\Gamma( ({p+\alpha+m+1})/2)}		
+|\sin(\alpha\pi)|\frac{p+\alpha+m}{p+1}\widetilde  C_{p,\alpha}
		\bigg\}^2
\bigg\}^{1/2}.
\end{split}
	\end{equation*}

For $\Proj_p^+$, using \eqref{LPiMinus-2} with the opposite sign leads us to
estimate the new combination, which can be expressed using \eqref{dUnend-0},
 \begin{equation}\label{EndPi+1}
 \begin{split}
\frac{\hat \omega'_{p}}{2p+1}-\frac{\hat \omega'_{p+1}}{2p+3} & =
 \frac{1}{2^{\alpha+m} \Gamma(\alpha+m)}
\int_{-1}^{1}
(1-\xi ^2)^{\alpha+m-1}
  \big({}^{r\!} G_{p-\alpha-m+1}^{(\alpha+m-1/2)}(\xi ) \\
  &- {}^{r\!} G_{p-\alpha-m+2}^{(\alpha+m-1/2)}(\xi ) \big)\,v^{(m)}(\xi ) \, {\rm d}\xi
+\sin(\alpha\pi)\sum_{i=0}^{m-1}(-1)^{p+i}
\breve{C}_{p,\alpha+i}
v^{(i)}(-1+),
 \end{split}
\end{equation}
 where
$$
\breve{C}_{p,\gamma} :
=\frac{2^{\gamma }(p+1)\Gamma(\gamma)}{\pi}\frac{\Gamma( p-\gamma+1)}{\Gamma(p+\gamma+2)}.
$$
Similarly,  using \cite[(4.30)]{Liu19MC-Optimal} yields
\begin{equation}\label{EndPi+2}	
  \frac{(1-\xi ^2)^{\alpha+m-1}}{2^{\alpha+m} \Gamma(\alpha+m)}\big|
{}^{r\!} G_{p-\alpha-m+1}^{(\alpha+m-1/2)}(\xi )
  - {}^{r\!} G_{p-\alpha-m+2}^{(\alpha+m-1/2)}(\xi ) \big| \le
\frac{1}{2^{\alpha+m-1} \sqrt{\pi}}
	\frac{ \Gamma( ({p-\alpha-m})/ 2+1)}{\Gamma( ({p+\alpha+m+1})/2)}.
\end{equation}
Similar to the derivation of
\eqref{BoundForUnEnd-1}, we can obtain
\begin{equation}\label{EndPi+3}
	\sum_{i=0}^{m-1}
\frac{\breve{C}_{p,\alpha+i} }
{\breve{C}_{p,\alpha} }
=1+	\sum_{i=1}^{m-1}
    \bigg( \prod_{r=1}^i\frac{2(\alpha+r-1)}{(p-\alpha-r+1)(p+\alpha+r+1)}\bigg)
<\frac{p+\alpha+m}{p+2}.
	\end{equation}
Thus, by \eqref{EndPi+1}-\eqref{EndPi+3},
\begin{equation}\label{EndPi+4}
		\begin{split}	
\Big(\frac{\hat \omega'_{p}}{2p+1}-\frac{\hat \omega'_{p+1}}{2p+3}\Big)^2 & \le \bigg\{\frac{1}{2^{\alpha+m-1} \sqrt{\pi}}
	\frac{ \Gamma( ({p-\alpha-m})/ 2+1)}{\Gamma( ({p+\alpha+m+1})/2)}
		\\
&\quad +|\sin(\alpha\pi)|\frac{p+\alpha+m}{p+2}\breve{C}_{p,\alpha}
		\bigg\}^2\big(\rFraSLNorm\big)^2.
		\end{split}
		\end{equation}
From \eqref{LPiMinus-2}, \eqref{FracEndL2B} and \eqref{EndPi+4}, we have
\begin{equation}\label{EndPi+5}
		\begin{split}
		\|\omega-\Proj_p^+  \omega\|_{L^2(\RefDomain)}\le&{\Upsilon}^{(\alpha,m)}_p\, \rFraSLNorm ,
		\end{split}
		\end{equation}	
where we denote
\begin{equation*}
\begin{split}
{\Upsilon}^{(\alpha,m)}_p&:=\bigg\{ \frac{4}{(2\alpha+2m-1) \pi}
		\frac{ \Gamma( {p-\alpha-m+2})}{\Gamma( {p+\alpha+m+1})}+ \frac{(2p+3)\,\Gamma^2( ({p-\alpha-m})/ 2+1)}{ 2^{2(\alpha+m)}{\pi}\,\Gamma^2( ({p+\alpha+m+3})/2) }\\
&\quad + |\sin(\alpha\pi)|\frac{2^{2(3\alpha+2)}\Gamma^2(\alpha+1)}{\pi^2(2\alpha+1)} \frac{(p+\alpha+m+1)^2}{({p-\alpha-m+3})^2}
\frac{p+3/2}{p+1}
\frac{\Gamma(2p-2\alpha+1)}{\Gamma(2p+2\alpha+3)}\\
&\quad +\frac{2}{2p+1}\bigg\{\frac{1}{2^{\alpha+m-1} \sqrt{\pi}}
	\frac{ \Gamma( ({p-\alpha-m})/ 2+1)}{\Gamma( ({p+\alpha+m+1})/2)}	
+|\sin(\alpha\pi)|\frac{p+\alpha+m}{p+2} \breve{C}_{p,\alpha}
		\bigg\}^2
\bigg\}^{1/2}.
\end{split}
\end{equation*}

In this above estimates, we tried to keep the estimates valid uniformly for small $p,$ so the bounds appear a bit involved.
To explicitly show the order in large $p,$
we use the formula (see e.g.,  \cite[(5.11.13)]{Olver2010Book}\rm): for $a<b$,
\begin{equation}\label{gamratioA}
\frac{\Gamma(z+a)}{\Gamma(z+b)}=z^{a-b}+\frac{1}{2}(a-b)(a+b-1)z^{a-b-1}+O(z^{a-b-2}),\quad z\gg 1,
\end{equation}	
and can obtain the desired estimates.
\end{proof}

% \begin{rem}
% Note that for endpoint singularities, $\alpha$ must not be an integer.
%  For functions with a singularity at the right endpoint, that is, $\omega\in {\mathcal U}^{\alpha,m}_{1-}(\RefDomain),$ Lemma \ref{EndPi} holds under the corresponding conditions, with $\rFraSRNorm$  in place of $\rFraSLNorm.$
% \end{rem}

\begin{rem}\label{rmk31A} It is seen from the above proof that the exact formula for the Legendre expansion coefficient in \eqref{Unend-0} is the starting point and critical for the analysis. In fact, if $\alpha\to 0,$ then it reduces to
\begin{equation}\label{Unend-a}
\hat{\omega}_n
= \frac{2n+1}{2^{m+1}m!}
\int_{-1}^{1} (1-\xi^2)^{m}\,
G_{n-m}^{(m+1/2)}(\xi)\,\omega^{(m)}(\xi)\,{\rm d}\xi,
\end{equation}
where in contrast to the fractional setting, the Gegenbauer polynomials
$\{G_{n-m}^{(m+1/2)}(\xi)\}$ are mutually orthogonal in $L^2_{\rho^m}(I_{\rm ref})$ with $\rho=1-\xi^2.$
Thus, if $\omega\in L^2_{\rho^m}(I_{\rm ref}),$  we can use the  Parseval's identity,  and \eqref{Unend-a} to derive
\begin{equation}\label{Unend-b0}
\int_{-1}^1 |\omega^{(m)}(\xi)|^2 (1-\xi^2)^m \,{\rm d}\xi
=\sum_{n=m}^{\infty}\frac{2(n+m)!}{(2n+1)(n-m)!}\,|\hat{\omega}_n|^2.
\end{equation}
Then we can estimate the $L^2$-projection error:
\begin{equation}\label{Unend-b1}
\begin{split}
\|\omega-\Proj_p \omega\|_{L^2(\RefDomain)}^{2}
&=
\sum_{n=p+1}^{\infty}\frac{2}{2n+1}\,|\hat{\omega}_n|^2
\leq
\max_{n\geq p+1}\Big\{\frac{(n-m)!}{(n+m)!}\Big\}
\sum_{n=p+1}^{\infty} \frac{2(n+m)!}{(2n+1)(n-m)!}\,|\hat{\omega}_n|^2
\\
&\leq
\frac{(p-m+1)!}{(p+m+1)!}
\int_{-1}^1 |\omega^{(m)}(\xi)|^2 (1-\xi^2)^m \,{\rm d}\xi\le C p^{-2m}
\|\omega^{(m)}\|^2_{L^2_{\rho^m}(\RefDomain)},
\end{split}
\end{equation}
to obtain analogous estimates as in \cite{Schwab1998Book,Shen2011Book}, albeit via a slightly different approach. We can then further estimate the Gauss–Radau projections using Lemma~\ref{LPiMinus}.
\qed
\end{rem}

\subsection{Proof of Theorem \ref{MainEstiA} }
\label{pfMainEstiA}
Using the new approximation results, we  estimate the four terms in
 \eqref{globalbound}-\eqref{ATcase}. Note that the Gauss-Radau projections
 $\pi^{ \pm} u|_{I_j}=\pi_{p_1}^{\pm} u|_{I_j}$ are assembled piece by piece,
 so we can estimate the four terms
 \eqref{globalbound}-\eqref{ATcase} separately for each $I_j.$ Here,
we focus on the first sub-interval
 $I_1=(a,x_1)$ with solution singularity, as the others sub-intervals with the given  regularity assumption  can be estimated as in  \cite{Castillo2002MC}.
 % For clarity, we split the proof into  four steps.
% % We start from the global bound \eqref{globalbound},
% % the desired estimates \eqref{MainII-a-d0} and \eqref{MainII-a-dp} will follow once we bound the four
% % projection error terms in \eqref{globalbound} by
% % $\VertV{u}_{\mathcal{E}, T}^{A}$.
% \medskip
% \noindent
% \emph{Step 1. Estimates on the first interval $I_1=(a,x_1)$.}
The linear transforms between $I_1$ and $I_{\rm ref}=(-1,1)$ are
\begin{equation}\label{map0}
x=\chi(\xi;I_1):=x_0+\frac{h_1}{2}(1+\xi);\;\;  \xi\in\RefDomain, \quad
\xi=\chi^{-1}(x;I_1)=2\frac{x-x_0}{h_1}-1,\;\;  x\in I_1.
\end{equation}
Set $\omega(\xi,t):=u(\chi(\xi;I_1),t)=u|_{I_1}\circ \chi,$ and   $\Proj_{p_1}^\pm\omega = (\GProj_{p_1}^\pm u|_{I_1})\circ\chi.$
% If $u(\cdot,t)|_{I_1}\in {\mathcal U}^{\alpha,m}_{a+}(I_1),$ we have
%  $\omega(\cdot,t)\in {\mathcal U}^{\alpha,m}_{-1+}(\RefDomain)$ and vice versa.
 We obtain from \eqref{EndPi-1-1} that
\begin{equation}\label{Cor:33-3a0}
\begin{split}
\|(u-\GProj_{p_1}^- u)(\cdot,t)\|_{I_1}&=\Big(\frac{h_1} 2\Big)^{1/2}\|\omega-\Proj_{p_1}^- \omega\|_{L^2(\RefDomain)}\\
&\le C
{h}_1 ^{1/2}   p_1^{-\min\{2\alpha+1,\alpha+m-1/2\}}\rFraSLNorm,
\end{split}\end{equation}
where by \eqref{defn:FracSpa-3-1},
$$\rFraSLNorm=\big\|\MD^{m}\big(\CFrD{\alpha}{-1+}   \omega\big)\big\|_{{L^1}(\RefDomain)} + \sum_{i=0}^{m-1}\big|\MD^i \big(\CFrD{\alpha}{-1+}   \omega\big) (-1+)\big|.$$
%Since $u(\cdot,t)\in{\mathcal U}^{\alpha,m}_{a+}(I_1)$ with $\alpha\in(k-1,k)$, we have
%$u\in W^{k,1}(I_1)$ and $\CFrD{\alpha}{a+}u\in W^{m,1}(I_1)$ by definition
%of ${\mathcal U}^{\alpha,m}_{a+}(I_1)$.
%Let us to estimates $\rFraSLNorm.$
Making a change of variables and using the definition of Caputo derivative, we find
\begin{equation}\label{Cor:33-6}
\begin{split}
\CFrD{\alpha}{-1+} \omega(\xi) & =
\Big({\mathcal I}_{-1+}^{k-\alpha} \frac{\rd ^{k}\omega(\xi)}{\rd\xi^{k}}\Big)
=\frac {1} {\Gamma(k-\alpha)}\int_{-1}^{\xi} (\xi-\tau)^{k-\alpha-1}\frac{\rd ^{k}\omega(\tau)}{\rd\tau^{k}}{\rm d}\tau\\
&=\Big(\frac{h_1}{2}\Big)^{\alpha}\frac{1} {\Gamma(k-\alpha)}\int_{x_0}^{x} (x-y)^{k-\alpha-1}\frac{\rd ^{k}u(y)}{\rd y^{k}}{\rm d}y\\
&
=\Big(\frac{h_1}{2}\Big)^{\alpha}\CFrD{\alpha}{a+}  u(x),
\end{split}
    \end{equation}
where
% $x = \chi(\xi; I_1),$ $ y = \chi(\tau; I_1)$
% and
we used
$$\rd\tau = \frac{2}{h_1}\rd y,\quad \frac{\rd ^{k}\omega(\tau)}{\rd\tau^{k}}=\Big(\frac{h_1}{2}\Big)^{k}\frac{\rd ^{k}u(y)}{\rd y^{k}},\quad (\xi-\tau)^{k-\alpha-1}= \Big(\frac{h_1}{2}\Big)^{1+\alpha-k} (x-y)^{k-\alpha-1}.  $$
We further take integer derivatives and find readily that for $1\le i\le m$,
\begin{equation}\label{Cor:33-7}
\begin{split}
\partial_\xi^i\CFrD{\alpha}{-1+} \omega(\xi) =\Big(\frac{h_1}{2}\Big)^{\alpha} \partial_\xi^i\CFrD{\alpha}{a+}  u(x)
=\Big(\frac{h_1}{2}\Big)^{\alpha+i}\partial_x^i\CFrD{\alpha}{a+}  u(x),\quad x\in I_1.
\end{split}
    \end{equation}
%Hence $\omega(\cdot,t)\in{\mathcal U}^{\alpha,m}_{-1+}(\RefDomain)$ and
Thus,  we have that for $0<h_1<2,$
\begin{equation*}
\begin{split}
U_{-1+}^{\alpha,m}(\omega;\RefDomain)
&=\big\|\MD^{m}\big(\CFrD{\alpha}{-1+}   \omega\big)\big\|_{{L^1}(\RefDomain)} + \sum_{i=0}^{m-1}\big|\MD^i \big(\CFrD{\alpha}{-1+}   \omega\big) (-1+)\big|
\\
&=\Big(\frac{h_1}{2}\Big)^{\alpha+m}\big\|\MD^{m}\big(\CFrD{\alpha}{a+}   u\big)\big\|_{{L^1}(I_1)} + \sum_{i=0}^{m-1}\Big(\frac{h_1}{2}\Big)^{\alpha+i}\big|\MD^i \big(\CFrD{\alpha}{a+}   u\big) (a+)\big|
\\
&=\Big(\frac{h_1}{2}\Big)^{\alpha}\Big\{ \Big(\frac{h_1}{2}\Big)^{m}\big\|\MD^{m}\big(\CFrD{\alpha}{a+}   u\big)\big\|_{{L^1}(I_1)} + \sum_{i=0}^{m-1}\Big(\frac{h_1}{2}\Big)^{i}\big|\MD^i \big(\CFrD{\alpha}{a+}   u\big) (a+)\big|\Big\}
\\
&\le \Big(\frac{h_1}{2}\Big)^{\alpha}U_{a+}^{\alpha,m}(u(\cdot,t);I_1).
\end{split}
    \end{equation*}
Thus, the estimate \eqref{Cor:33-3a0} becomes
\begin{equation}\label{Th45-1}
\begin{split}
\|(u-\GProj_{p_1}^- u)(\cdot,t)\|_{I_1}\le C
h_1 ^{\alpha+1/2}   p_1^{-\min\{2\alpha+1,\alpha+m-1/2\}}U_{a+}^{\alpha,m}(u(\cdot,t);I_1).
\end{split}\end{equation}
\begin{comment}
Define $\GProj_{p_1}^\pm u$ on $I_1$ by $(\GProj_{p_1}^\pm u)(x,t):=(\Proj_{p_1}^\pm \omega)(\chi^{-1}(x;I_1),t),x\in I_1 .$
Equivalently, $\Proj_{p_1}^\pm\omega = (\GProj_{p_1}^\pm u)\circ\chi$ on $\RefDomain$.
Using $\rd x=(h_1/2)\rd\xi$, we have
\begin{equation}\label{Cor:33-3a}
\|\omega-\Proj_{p_1}^\pm \omega\|_{L^2(\RefDomain)}
=\Big(\frac{2}{h_1}\Big)^{1/2}
\|(u-\GProj_{p_1}^\pm u)(\cdot,t)\|_{I_1}.
\end{equation}
Combining \eqref{EndPi-1-1} with $p=p_1$, together with \eqref{Cor:33-3a} and the norm scaling, yields
\begin{equation}\label{Th45-1}
\|\GProj^{-}_{p_1} u(\cdot,t)-u(\cdot,t)\|_{I_1}
\le C\,h_1^{\alpha+1/2}
p_1^{-\min\{2\alpha+1,\alpha+m-1/2\}}
U^{\alpha,m}_{a+}(u(\cdot,t);I_1).
\end{equation}
\end{comment}
Similarly, we can show
\begin{equation}\label{Th45-2}
\big\|(\GProj^{-}_{p_1} u_t-u_t)(\cdot,t)\big\|_{I_{1}} \leq  C h_1^{\alpha+1/2}
 p_1^{-\min\{2\alpha+1,\alpha+m-1/2\}}
 \,U^{\alpha, m}_{a+}({u_t(\cdot,t);I_1})  .
\end{equation}

(i) When $d=0$, all terms involving $q$ in \eqref{globalbound} are absent.
Thus, using \eqref{Th45-1} and \eqref{Th45-2} yields
\begin{equation}\label{err-d0-1}
\begin{aligned}
&\sup_{0\le t\le T}\|(\GProj^{-}u-u)(\cdot,t)\|_{I_1}
+\int_{0}^{T}\|(\GProj^{-}u_t-u_t)(\cdot,t)\|_{I_1}\,\rd t
+\|\GProj^{-}u_{\rm ic}-u_{\rm ic}\|_{I_1}
\\
&\qquad\le
C
\frac{h_1^{\alpha+1/2}}{p_1^{\min\{2\alpha+1,\alpha+m-1/2\}}}
\Big(\sup_{0\le t\le T}U^{\alpha,m}_{a+}(u(\cdot,t);I_1)
+\int_{0}^{T}U^{\alpha,m}_{a+}(u(\cdot,t);I_1)\,\rd t\Big),
\end{aligned}
\end{equation}
which leads to the contribution of the errors in the first interval for the desired estimate in \eqref{MainII-a-d0}.

\medskip

{(ii) When $d\neq0$,}
the diffusion is present, so we have to use the approximation result on $\pi^+ q|_{I_1}$ with $q=\sqrt{d}\,u_x.$
Since $u\in W^{k,1}(I_1)$, we have $u_x\in W^{k-1,1}(I_1)$, and
$$
\CFrD{\alpha-1}{a+}(u_x)
   ={\mathcal I}_{a+}^{(k-1)-(\alpha-1)}(u_x)^{(k-1)}
   ={\mathcal I}_{a+}^{k-\alpha}u^{(k)}
   =\CFrD{\alpha}{a+}u\in W^{m,1}(I_1).
$$
% Consequently,
% $$
% \MD^i(\CFrD{\alpha-1}{a+}u_x)(a+)
%    =\MD^i(\CFrD{\alpha}{a+}u)(a+),\quad 0\le i\le m-1,
% $$
% and these traces are finite by assumption.
Therefore,
$u_x(\cdot,t)\in{\mathcal U}^{\alpha-1,m}_{a+}(I_1)$.
Using the same argument with $\alpha$ replaced by $\alpha-1$ as for \eqref{Th45-1} and
the estimate for $\Proj^+_{p_1}$  in Theorem~\ref{EndPi}, we obtain
\begin{equation}\label{Th45-3}
\|(\GProj^{+}_{p_1} q-q)(\cdot,t)\|_{I_1}
\le C\,\sqrt d\,h_1^{\alpha-1/2}
p_1^{-\min\{2\alpha-3/2,\alpha+m-3/2\}}
U^{\alpha,m}_{a+}(u(\cdot,t);I_1).
\end{equation}
With the aid of the above relevant estimates, we can obtain
\begin{equation}\label{err-dp-1}
\begin{aligned}
&\sup_{0\le t\le T}\|(\GProj^{-}u-u)(\cdot,t)\|_{I_1}
+\int_{0}^{T}\|(\GProj^{-}u_t-u_t)(\cdot,t)\|_{I_1}\,\rd t
+\|\GProj^{+}q-q\|_{L^2(I_1\times(0,T))}
\\
&\qquad\quad
+\|\GProj^{-}u_{\rm ic}-u_{\rm ic}\|_{I_1}
\le
C
\frac{h^{\alpha-1/2}}{p^{\min\{2\alpha-3/2,\alpha+m-3/2\}}}
\Big(\sup_{0\le t\le T}U^{\alpha,m}_{a+}(u(\cdot,t);I_1)
\\
&\qquad\quad\qquad\quad\qquad\quad+\int_{0}^{T}U^{\alpha,m}_{a+}(u(\cdot,t);I_1)\,\rd t
+\sqrt{d}\,\|U^{\alpha,m}_{a+}(u(\cdot,t);I_1)\|_{L^2(0,T)}\Big),
\end{aligned}
\end{equation}
which yields the error bound in \eqref{MainII-a-dp} on $I_1.$

\vspace{0.3em}
%\noindent\emph{Step 2. Estimates on the remaining intervals $I_j$, $j=2,\dots,M$.}\,

For the other sub-intervals $I_j,~ j=2,\dots,M$, we assumed the  solution has a usual Sobolev regularity and can use     the approximation results of the Gauss-Radau projections as in  \cite{Castillo2002MC}
  (also see Remark \ref{rmk31A}) derive that for $p\ge s,$
  $$
\|\omega-\Proj_p^\pm  \omega\|_{L^2(\RefDomain)} \leq C
\, p ^{-{s}-1} \|\omega^{({s}+1)}\|_{L^2(\RefDomain)},
$$
and
\begin{equation} \label{Cor00}
|\omega(\pm1)-\Proj_p^\pm  \omega(\pm1)|  \leq C\,  p ^{-{s}-1/2} \|\omega^{({s}+1)}\|_{L^2(\RefDomain)},
\end{equation}
where  $C$ depends on $s$ but is independent of $p$ and $\omega$.  Then using a similar re-scaling argument as before, we can obtain the order   $ {\mathcal O}(\tfrac{h^{s+1}}{p^{s+1}})$ (for $d=0$) and  $ {\mathcal O}(\tfrac{h^{s}}{p^{s}})$ (for $d\not=0$) in   \eqref{MainII-a-d0}-\eqref{MainII-a-dp}. Here, we omit the details, and complete the proof.

{\color{red}

}

\section{Extensions to interior singularities:
 Fitted and unfitted cases}
\label{two cases}
\setcounter{equation}{0}
\setcounter{lmm}{0}
\setcounter{thm}{0}

In this section, we  extend
the  analysis of the LDG scheme
to singular solutions with an interior singularity, say at $x=\theta\in (a,b).$
We show the different convergence behaviors for two cases: $\theta$
 aligned with the grid (fitted case) or located inside a sub-interval $I_j$ (unfitted case).

% Motivated by the numerical evidence presented in Figure~\ref{Figs_2}, we study the
% $p$--version error behavior of the LDG method for solutions possessing
% singularities at different spatial locations.
% Two cases are considered: a singularity at the node $x_1$ separating $I_1$ and $I_2$,
% and a singularity lying strictly inside $I_1$.
% These cases require fractional regularity descriptions beyond left endpoint
% singularities.

% In the fitted setting where the singularity is located at a mesh node $x_j$ ($1\le j\le M-1$) and the solution satisfies
% $u(x,t)=\mathcal{O}(|x-x_j|^\alpha)$ with a non-vanishing leading coefficient $d\neq 0$,
% one can follow the same lines as in the proof of Theorem~\ref{MainEstiA} to obtain the same $p$--version convergence behavior.
% In this case,  the spaces involve the fractional  derivative operator: $\CFrD{\alpha}{1-}$  and the corresponding approximation results can be derived by a simple transformation.
% In this subsection, we analyze the fitted but degenerate case $d=0$, where the leading singular term cancels and the $p$--version exhibits a different convergence behavior.
\subsection{Fitted case}
For simplicity, we assume that the singular point is at $\theta=x_1.$ In this case, we need to involve the Gauss-Radau projections to handle both left- and right- (``end-point'') singularities. Note that on
 $I_2=(x_1,x_2),$ the singularity is of left-endpoint type which can be described by
\eqref{defn:FracSpa-AB}--\eqref{defn:FracSpa-3-1}. While on $I_1=(a,x_1),$ it becomes a right-endpoint
singularity, so we just replace
$\CFrD{\alpha}{\hat a+}$ by $\CFrD{\alpha}{\hat b-}$ in
\eqref{defn:FracSpa-AB}--\eqref{defn:FracSpa-3-1}:
for $\alpha\in (k-1,k)$ and $k,m\in\mathbb N$, we define the space
\begin{equation}\label{defn:FracSpa-AB-right}
{\mathcal U}_{\hat b-}^{\alpha, m}(\Lambda)
:=\big\{u\,:\, u\in W^{k,1}(\Lambda)\ \ {\rm and}\ \
\CFrD{\alpha}{\hat b-}u \in W^{m,1}(\Lambda)\big\},
\end{equation}
and correspondingly, we denote
\begin{equation}\label{defn:FracSpa-3-1-right}
{U_{\hat b-}^{\alpha, m}}({u;\Lambda})
:=\|\MD^{m}(\CFrD{\alpha}{\hat b-}u)\|_{L^1(\Lambda)}
+ \sum_{i=0}^{m-1}|\MD^i (\CFrD{\alpha}{\hat b-}u) (\hat b-)|.
\end{equation}

By an argument analogous to that  for Theorem~\ref{EndPi}, we  can obtain the following
Gauss--Radau projection error estimates for the case of the right-endpoint singularity.
\begin{lemma}\label{EndPiII}
Let $\alpha \in (k-1, k)$ and $k,m\in \mathbb N.$ Then for any
 $\omega\in {\mathcal U}^{\alpha,m}_{1-}(\RefDomain)$ with $1/2< \alpha+m<  p+1,$ we have
\begin{equation}\label{EndPiR-1-1}
		\begin{split}
		\|\omega-\Proj_p^-  \omega\|_{L^2(\RefDomain)}\le& C p^{-\min\{2\alpha+1/2,\alpha+m-1/2\}}\rFraSRNorm,
		\end{split}
		\end{equation}
        \begin{equation}\label{REndPiII-1-1}
		\begin{split}
		\|\omega-\Proj_p^+ \omega\|_{L^2(\RefDomain)}\le& C p^{-\min\{2\alpha+1,\alpha+m-1/2\}}\rFraSRNorm,
		\end{split}
		\end{equation}
and
\begin{equation}\label{EndPiR-2-1}
		\begin{split}
		|\omega(\pm1)-\Proj_p^\pm  \omega(\pm1)|\le Cp^{-\min\{2\alpha,\alpha+m-1/2\}}\rFraSRNorm,
		\end{split}
		\end{equation}
where $C$ is a positive constant independent of $p$ and $\omega$.
\end{lemma}

A standard scaling argument, as used in the proof of Theorem~\ref{MainEstiA}
(see Section~\ref{pfMainEstiA}), transfers the above estimates from the reference
element to a general element. This yields the following Corollary.

\begin{cor}\label{Cor:34}
Assume that the singular point is at $\theta=x_1,$ and that
 $\tilde\omega|_{I_{1}}\in {\mathcal U}^{\alpha,m}_{x_{1}-}(I_{1})$  with $ \alpha > 0,$ $0<h\le 2,$   and $ \alpha+m< p+1.$
Then the following estimates hold:
\begin{equation*}\label{Cor:34-1a}
		\begin{split}
\big\|\GProj^{-}_{p} \tilde\omega-\tilde\omega\big\|_{I_{1}} \leq &  C h^{\alpha+1/2}p^{-\min\{2\alpha+1/2,\alpha+m-1/2\}}\, \rFraMiSRNorm ,
	\end{split}
		\end{equation*}	
        \begin{equation*}\label{Cor:34-1b}
		\begin{split}
\big\|\GProj^{+}_{p} \tilde\omega-\tilde\omega\big\|_{I_{1}} \leq &  C h^{\alpha+1/2}p^{-\min\{2\alpha+1,\alpha+m-1/2\}}\, \rFraMiSRNorm ,
	\end{split}
		\end{equation*}	
and
\begin{equation*}\label{Cor:34-2}
		\begin{split}
\big|\big(\GProj^{+}_{p} \tilde\omega-\tilde\omega\big)\big(x_{1}\big)\big| & \leq C h^{\alpha}p^{-\min\{2\alpha,\alpha+m-1/2\}}\,\rFraMiSRNorm ,
\end{split}
		\end{equation*}
\begin{equation*}\label{Cor:34-3}
		\begin{split}	
\big|\big(\GProj^{-}_{p} \tilde\omega-\tilde\omega\big)\big(a\big)\big| & \leq
C h^{\alpha}p^{-\min\{2\alpha,\alpha+m-1/2\}}\rFraMiSRNorm ,
\end{split}
		\end{equation*}
where $C$ depends on $\alpha$ but is independent of $p$ and $\tilde\omega$.
\end{cor}

% Now, we are ready to present our optimal convergence result for the class of solutions with a singularity at $x_1$.
Define the broken Sobolev space associated to the mesh $\mathcal T$
as an  extension of \eqref{SpaceL0}–\eqref{NormL0} for singular solutions on fitted grids:
$$
{\mathcal W}^{\alpha,m,s}_{x_1}(I;\mathcal T)
:=\big\{ v:I\to\mathbb{R}:\ v|_{I_1}\in {\mathcal U}^{\alpha,m}_{x_1-}(I_1),\
v|_{I_2}\in {\mathcal U}^{\alpha,m}_{x_1+}(I_2),\
v|_{I_j}\in H^{s+1}(I_j),\ j=3,\ldots,M\big\},
$$
and denote
\begin{equation*}
\begin{split}
\|u\|_{{\mathcal W}^{\alpha,m,s}_{x_1}(I;\mathcal T)}
:=
{U_{x_1-}^{\alpha,m}}(u;I_1)
+{U_{x_1+}^{\alpha,m}}(u;I_2)
+\sum_{j=3}^{M}\|u^{(s+1)}\|_{L^2(I_j)}.
\end{split}
\end{equation*}
\begin{comment}
\begin{thm}\label{MainEstiB}
Let $u$ be the solution of \eqref{CDEq} and
$q=\sqrt{d}\,u_x,$ and let $(u_{h}^{p},q_{h}^{p})^{\T}$ be the LDG solution of
\eqref{LDG-1}–\eqref{ubsUB} with the regularity $\VertV{u}_{\mathcal{E}, T}<\infty$, where   $\VertV{u}_{\mathcal{E}, T}$ denotes the norm  in \eqref{NewNormA}, but  $\|\cdot\|_{\mathcal W}$ is replaced by
$\|\cdot\|_{{\mathcal W}^{\alpha,m,s}_{x_1}(I;\mathcal T)}$.
Then for $h=\max\{h_j\}$ and $p=\min\{p_j\}$ with $p>\alpha+m-1$, we have
\begin{itemize}
\item[(i)]    for $d=0$ and $p\ge s,$
\begin{equation}\label{MainII-b-d0}
\|(u-u_{h}^{p})(\cdot,T)\|
\le
C\Big(
\frac{h^{\alpha+1/2}}{p^{\min\{2\alpha+1/2,\alpha+m-1/2\}}}
+\frac{h^{s+1}}{p^{s+1}}
\Big)
\VertV{u}_{\mathcal{E},T};
\end{equation}
\item [(ii)]  for $d\neq0$
and $p\ge s-1,$
\begin{equation}\label{MainII-b-dp}
\|(u-u_{h}^{p})(\cdot,T)\|
+\|q-q_{h}^{p}\|_{Q_T}
\le
C\Big(
\frac{h^{\alpha-1/2}}{p^{\min\{2\alpha-3/2,\alpha+m-3/2\}}}
+\frac{h^{s}}{p^{s}}
\Big)
\VertV{u}_{\mathcal{E},T},
\end{equation}
\end{itemize}
where the constant
$C$ depends on $\alpha, s,$  but is independent of $h,$ $p$ and $u$.
\end{thm}
\end{comment}
\begin{thm}\label{MainEstiB}
Let $u$ be the solution of \eqref{CDEq} and $q=\sqrt{d}\,u_x$, and let $(u_h^p,q_h^p)^{\T}$
be the LDG solution of \eqref{LDG-1}--\eqref{ubsUB} with the regularity
$\VertV{u}_{\mathcal{E},T}<\infty$, where $\VertV{u}_{\mathcal{E},T}$ denotes the norm in
\eqref{NewNormA}, but $\|\cdot\|_{\mathcal W}$ is replaced by
$\|\cdot\|_{{\mathcal W}^{\alpha,m,s}_{x_1}(I;\mathcal T)}$.
Then for $h=\max\{h_j\}$ and $p=\min\{p_j\}$ with $p>\alpha+m-1$, we have, for $d=0$ and $p\ge s$,
\begin{equation}\label{MainII-b-d0}
\|(u-u_h^p)(\cdot,T)\|
\le
C\Big(
\frac{h^{\alpha+1/2}}{p^{\min\{2\alpha+1/2,\alpha+m-1/2\}}}
+\frac{h^{s+1}}{p^{s+1}}
\Big)\VertV{u}_{\mathcal{E},T}.
\end{equation}
For $d\neq 0$ (and $p\ge s-1$), under the same regularity assumption $\VertV{u}_{\mathcal{E},T}<\infty$,
the diffusion case admits the same $h$-- and $p$--convergence orders as in Theorem~\ref{MainEstiA} {\rm (ii)}.
Here the constant $C$ depends on $\alpha$ and $s$, but is independent of $h$, $p$, and $u$.
\end{thm}
\begin{proof}
The proof is a direct extension of that of Theorem~\ref{MainEstiA} to the fitted interior singularity at the mesh node $x_1$.
As in Subsection~\ref{pfMainEstiA}, we estimate the terms in \eqref{globalbound}--\eqref{ATcase} by elementwise Gauss--Radau projections
$\pi^\pm u|_{I_j}=\pi_{p_j}^\pm(u|_{I_j})$, and the only new point is to control the two elements adjacent to $x_1$, namely
$I_1=(a,x_1)$ and $I_2=(x_1,x_2)$.
On $I_1$ the singularity is attached to the \emph{right} endpoint $x_1$, hence
$u(\cdot,t)|_{I_1}\in{\mathcal U}^{\alpha,m}_{x_1-}(I_1)$ and Corollary~\ref{Cor:34} yields, for $0\le t\le T$,
\begin{equation}\label{BI1-u2}
\|(u-\GProj^-_{p_1}u)(\cdot,t)\|_{I_1}
\le C\,h_1^{\alpha+1/2}\,p_1^{-\min\{2\alpha+1/2,\alpha+m-1/2\}}\,
U^{\alpha,m}_{x_1-}(u(\cdot,t);I_1),
\end{equation}
and the same bound holds with $u$ replaced by $u_t$.
On $I_2$ the singularity is attached to the \emph{left} endpoint $x_1$, so
$u(\cdot,t)|_{I_2}\in{\mathcal U}^{\alpha,m}_{x_1+}(I_2)$ and, by the same scaling argument as in
Subsection~\ref{pfMainEstiA}, we obtain
\begin{equation}\label{BI2-u2}
\|(u-\GProj^-_{p_2}u)(\cdot,t)\|_{I_2}
\le C\,h_2^{\alpha+1/2}\,p_2^{-\min\{2\alpha+1,\alpha+m-1/2\}}\,
U^{\alpha,m}_{x_1+}(u(\cdot,t);I_2),
\end{equation}
and again the same bound holds with $u$ replaced by $u_t$.
\smallskip

(i)
When $d=0$, all terms involving $q$ in \eqref{globalbound} are absent.
Substituting the above projection estimates on $I_1$ and $I_2$ into
\eqref{globalbound}--\eqref{ATcase} yields the singular contribution in \eqref{MainII-b-d0}.
Combining this with the standard estimates on $I_j$, $j=3,\dots,M$, as in Theorem~\ref{MainEstiA},
completes the proof of \eqref{MainII-b-d0}.

\smallskip

(ii) When $d\neq 0$,
we set $q=\sqrt{d}\,u_x$. The argument is the same extension of Theorem~\ref{MainEstiA} (ii) to the present fitted setting:
the $h$-- and $p$--orders remain the same (with the norm $\|\cdot\|_{\mathcal W}$ replaced by
$\|\cdot\|_{{\mathcal W}^{\alpha,m,s}_{x_1}(I;\mathcal T)}$). Hence we omit the details.
\end{proof}

% We next consider the case where $u$ possesses an interior singularity lying strictly
% inside the first element $I_1=(a,x_1)$.

\subsection{Unfitted case}
It is known that if the location of the singularity does not fit the grid, the convergence order is lower than that of the fitted case. Typically, a loss of at least the  double convergence in $p$ for end-point singularities will occur.
% % Indeed,  we can observe from
% % Figure~\ref{Figs_1} and Figure \ref{Figs3_3} such a loss.
% In this subsection, we extend the framework and idea to study this unfitted case.
To fix the idea, we assume the location of the singularity is at $x=\theta\in I_1
=(a, x_1).$
%For a fixed $\theta\in I_1$, we
Denote
$I_{\theta}^{-}:=(a,\theta)$ and $I_{\theta}^{+}:=(\theta,x_1)$. To optimally characterize this type of singularities (e.g., $u(x,t)=|x-\theta|^\alpha t$), we first extend the fractional  framework as in \eqref{defn:FracSpa-AB}-\eqref{defn:FracSpa-3-1} to this setting. More precisely,
for $\alpha\in (k-1,k)$ with an integer $k\ge 1$, we define the fractional space
\begin{equation}\label{defn:FracSpa-int}
{\mathcal U}_{\theta}^{\alpha}(I_1)
:=\big\{u\,:\, u\in W^{k,1}(I_1),\
\CFrD{\alpha}{\theta-}u \in W^{1,1}(I_{\theta}^{-})
\ \text{and}\
\CFrD{\alpha}{\theta+}u \in W^{1,1}(I_{\theta}^{+})\big\},
\end{equation}
associated with the semi-norm:
\begin{equation}\label{defn:seminorm-int}
{U_{\theta}^{\alpha}}(u;I_1)
:=\big\|\MD\CFrD{\alpha}{\theta-}u\big\|_{L^1(I_{\theta}^{-})}
+\big\|\MD\CFrD{\alpha}{\theta+}u\big\|_{L^1(I_{\theta}^{+})}
+\big|\CFrD{\alpha}{\theta-}u(\theta^-)\big|
+\big|\CFrD{\alpha}{\theta+}u(\theta^+)\big|.
\end{equation}

% With the singularity fixed at $x=\theta_1\in I_1=(a,x_1)$, we present the optimal Gauss--Radau projection estimates for this unfitted case.
The proof of the following theorem on Gauss–Radau projections is very similar to that of Theorem~\ref{EndPi}, so we only outline the essentials to avoid repetition.
\begin{lemma}\label{ErrorLPiMinus}
 Let $\hat \theta=\chi^{-1}(\theta; I_1)\in \RefDomain$ be the mapped location of singularity via \eqref{map0}.
Then for any
$\omega\in{\mathcal U}^{\alpha}_{\hat \theta}(\RefDomain)$ with $0<\alpha\le p+1$, we have
\begin{equation}\label{ErrorLPiMinus-0}
\begin{split}
\|\omega-\Proj_p^\pm  \omega\|_{L^2(\RefDomain)}\le Cp^{-\alpha-1/2} \rFraSNorm,
\end{split}
\end{equation}
and
\begin{equation}\label{ErrorLPiMinus-1}
\begin{split}
|\omega(\pm1)-\Proj_p^\pm  \omega(\pm1)|\le Cp^{-\alpha-1/2} \rFraSNorm,
\end{split}
\end{equation}
where $C$ depends on $\alpha$ but is independent of $p$ and $\omega$.
\end{lemma}
\begin{proof}
Recall from \cite[(3.30) and Cor.~3.1]{Liu2021ACMOptimal} that for each $n\ge \alpha$,
the Legendre expansion coefficient satisfies
\begin{equation}\label{HatUnCaseC0}
\begin{split}
|\hat {\omega}_n| \le
\frac{(2n+1)}{2^{\alpha+2} \sqrt{\pi}}
\frac{\Gamma( ({n-\alpha})/ 2+1)}{ \Gamma( ({n+\alpha}+3)/2)}
\,\rFraSNorm .
\end{split}
\end{equation}
By an argument analogous to that in \cite[(3.30)]{Liu2021ACMOptimal}, we obtain
the following estimate for the derivative $\omega'(\xi)$,
\begin{equation}\label{duHatUn}
\begin{split}
|\hat \omega'_n| \le
\frac{(2n+1)}{ 2^{\alpha+1} \sqrt{\pi}}
\frac{\Gamma( ({n-\alpha}+1)/ 2)}{\Gamma( ({n+\alpha})/2+1)} \,\rFraSNorm .
\end{split}
\end{equation}
By an argument analogous to that used in the proof of Theorem~\ref{EndPi},
together with \eqref{HatUnCaseC0} and \eqref{duHatUn}, we obtain
\eqref{ErrorLPiMinus-0}.
\end{proof}

A standard scaling argument, as used in the proof of Theorem~\ref{MainEstiA}
(see Subsection~\ref{pfMainEstiA}), extends the above estimates from the reference
element to a general element. This yields the following Corollary.
\begin{cor} \label{CorII:31}
Assume that the interior singularity is located at $x=\theta \in I_1 = (a,x_1)$,
and that $\tilde\omega|_{I_1} \in {\mathcal U}^{\alpha}_{\theta}(I_1)$,
with $0 \le \alpha \le p+1$.
Then the following estimates hold:
\begin{equation}\label{CorII:31-01}
\big\|\GProj^{\pm}_{p}  \tilde\omega- \tilde\omega\big\|_{I_{1}} \leq C
h^{\alpha+1/2}
p^{-\alpha-1/2}
\, U_{\theta}^{\alpha}({\tilde\omega;I_{1}}),
  \end{equation}
and
\begin{subequations}\label{CorII:31-02}
\begin{align}
\big|\big(\GProj^{+}_{p}  \tilde\omega- \tilde\omega\big)\big(x_{1}\big)\big| & \leq C h^{\alpha}
p^{-\alpha-1/2}  \,U_{\theta}^{\alpha}({\tilde\omega;I_{1}}), \label{CorII:31-02a}\\
\big|\big(\GProj^{-}_{p}  \tilde\omega- \tilde\omega\big)\big(a\big)\big| & \leq Ch^{\alpha}
p^{-\alpha-1/2}\,U_{\theta}^{\alpha}({\tilde\omega;I_{1}}), \label{CorII:31-02b}
\end{align}
\end{subequations}	
where $C$ depends on $\alpha$ but is independent of $p$ and $\tilde\omega$.
\end{cor}

Now, we are ready to present our optimal convergence result for the class of solutions with an interior singularity located at $\theta\in I_1$.
Define the broken Sobolev space associated to the mesh $\mathcal T$
as an  extension of \eqref{SpaceL0}–\eqref{NormL0}
for singular solutions on unfitted grids:
$$
{\mathcal W}^{\alpha, s}_{\theta}(I;\mathcal{T}):=\big\{v: I\to \mathbb{R}\,:\, v|_{I_{1}} \in  {\mathcal U_{\theta}^{\alpha}}(I_1),\;  v|_{I_j}\in H^{s+1}(I_{j}),\; j=2, \ldots, M\big\},
$$
and denote
% \begin{equation*}
% 		\begin{split}
%  \|u\|_{\mathcal W_B}:=\|u\|_{{\mathcal W}^{\alpha,  s}_{\theta_1}(I;\mathcal{T})}:= {U_{\theta_1}^{\alpha}} (u, {I_1})
%   + \sum_{j=2 }^M \|u^{(s+1)}\|_{L^2(I_{j})}.
% \end{split}
% 		\end{equation*}
\begin{equation*}\label{NewNormInt}
\begin{split}
\|u\|_{{\mathcal W}^{\alpha,  s}_{\theta}(I;\mathcal{T})}:= {U_{\theta}^{\alpha}} (u;{I_1})
  + \sum_{j=2 }^M \|u^{(s+1)}\|_{L^2(I_{j})}.
\end{split}
\end{equation*}

\begin{thm}\label{MainEstiC}
Let $u$ be the solution of \eqref{CDEq} and
$q=\sqrt{d}\,u_x,$ and let $(u_{h}^{p},q_{h}^{p})^{\T}$ be the LDG solution of
\eqref{LDG-1}–\eqref{ubsUB} with the regularity $\VertV{u}_{\mathcal{E}, T}<\infty$, where   $\VertV{u}_{\mathcal{E}, T}$ denotes the norm  in \eqref{NewNormA}, but  $\|\cdot\|_{\mathcal W}$ is replaced by $\|\cdot\|_{{\mathcal W}^{\alpha,  s}_{\theta}(I;\mathcal{T})}$.
%Denote and assume
%\begin{equation}\label{NewNormC}
% \VertV{u}_{\mathcal{E}, T}^{B}:=
%2\sup _{0\le t\le T}\|u(\cdot, t)\|_{\mathcal W_B}
%+
%3\sqrt{d}\Big(\int_{0}^{T} \|u(\cdot, t)\|_{\mathcal W_B} \rd t\Big)^{1/2}
%+
%\int_{0}^{T} \|u_t(\cdot, t)\|_{\mathcal W_B} \rd t <\infty.
%\end{equation}
Then for $h=\max\{h_j\}$ and $p=\min\{p_j\}$ with $p>\alpha-1$, we have
\begin{itemize}
\item[(i)]    for $d=0$ and $p\ge s,$
\begin{equation}\label{MainII-c-d0}
\|(u-u_{h}^{p})(\cdot,T)\|
\le
C\Big(
\frac{h^{\alpha+1/2}}{p^{\alpha+1/2}}
+\frac{h^{s+1}}{p^{s+1}}
\Big)
\VertV{u}_{\mathcal{E},T};
\end{equation}
\item [(ii)]  for $d\neq0$ and $p\ge s-1,$
\begin{equation}\label{MainII-c-dp}
\|(u-u_{h}^{p})(\cdot,T)\|
+\|q-q_{h}^{p}\|_{Q_T}
\le
C\Big(
\frac{h^{\alpha-1/2}}{p^{\alpha-1/2}}
+\frac{h^{s}}{p^{s}}
\Big)
\VertV{u}_{\mathcal{E},T},
\end{equation}
\end{itemize}
where the constant
$C$ depends on $\alpha, s,$  but is independent of $h,$ $p$ and $u$.
\end{thm}
\begin{proof}
The proof follows the lines of Theorem~\ref{MainEstiA}.
Using the  approximation results in Lemma \ref{ErrorLPiMinus}, we also estimate the terms in \eqref{globalbound}--\eqref{ATcase} via the Gauss-Radau projections
 $\pi^{ \pm} u|_{I_j}=\pi_{p_1}^{\pm} u|_{I_j}$ with the emphasis on
  the first sub-interval
 $I_1=(a,x_1).$
Indeed, we obtain from Theorem~\ref{ErrorLPiMinus} and \eqref{CorII:31-01} that
\begin{equation}\label{unfit0}
\begin{split}
\|(u-\GProj_{p_1}^-u)(\cdot,t)\|_{I_1}
&=\Big(\frac{h_1}{2}\Big)^{1/2}\|\omega-\Proj_{p_1}^-\omega\|_{L^2(I_{\rm ref})}
\\
&\le C\,h_1^{1/2}\,p_1^{-\alpha-1/2}\,\rFraSNorm
\\
&\le C\,h_1^{\alpha+1/2}\,p_1^{-\alpha-1/2}\,U_{\theta}^{\alpha}(u(\cdot,t);I_1),
\end{split}
\end{equation}
where we used for $0<h_1<2,$
\begin{equation*}
\rFraSNorm
\le \Big(\frac{h_1}{2}\Big)^{\alpha}\,U_{\theta}^{\alpha}(u(\cdot,t);I_1).
\end{equation*}
The same argument applied to $u_t$ yields
\begin{equation}\label{Intut0}
\|( \GProj_{p_1}^-u_t-u_t)(\cdot,t)\|_{I_1}
\le C\,h_1^{\alpha+1/2}\,p_1^{-\alpha-1/2}\,U_{\theta}^{\alpha}(u_t(\cdot,t);I_1).
\end{equation}

(i) When $d=0$, all terms involving $q$ in \eqref{globalbound} are absent.
Thus, using \eqref{Th45-1} and \eqref{Th45-2} yields
\begin{equation}\label{interr-d0-1}
\begin{aligned}
&\sup_{0\le t\le T}\|(\GProj^{-}u-u)(\cdot,t)\|_{I_1}
+\int_{0}^{T}\|(\GProj^{-}u_t-u_t)(\cdot,t)\|_{I_1}\,\rd t
+\|\GProj^{-}u_{\rm ic}-u_{\rm ic}\|_{I_1}
\\
&\qquad\le
C
\frac{h_1^{\alpha+1/2}}{p_1^{\alpha+1/2}}
\Big(\sup_{0\le t\le T}U^{\alpha}_{\theta}(u(\cdot,t);I_1)
+\int_{0}^{T}U^{\alpha}_{\theta}(u(\cdot,t);I_1)\,\rd t\Big),
\end{aligned}
\end{equation}
which leads to the contribution of the errors in the first interval for the desired estimate in \eqref{MainII-c-d0}.

\medskip

{(ii) When $d\neq0$,}
set $q=\sqrt{d}\,u_x$ and the proof follows the same lines as in Theorem~\ref{MainEstiA} (ii).
Using \eqref{ErrorLPiMinus-0}, we obtain that
\begin{equation}\label{interr-dp-0}
\|(\GProj^{+}_{p_1} q-q)(\cdot,t)\|_{I_1}
\le C\,\sqrt d\,h_1^{\alpha-1/2}
p_1^{-\alpha+1/2}
U^{\alpha}_{\theta}(u(\cdot,t);I_1).
\end{equation}
With the aid of the above relevant estimates, we can obtain
\begin{equation}\label{interr-dp-1}
\begin{aligned}
&\sup_{0\le t\le T}\|(\GProj^{-}u-u)(\cdot,t)\|_{I_1}
+\int_{0}^{T}\|(\GProj^{-}u_t-u_t)(\cdot,t)\|_{I_1}\,\rd t
+\|\GProj^{+}q-q\|_{L^2(I_1\times(0,T))}
\\
&\qquad\quad
+\|\GProj^{-}u_{\rm ic}-u_{\rm ic}\|_{I_1}
\le
C
\frac{h^{\alpha-1/2}}{p^{\alpha-1/2}}
\Big(\sup_{0\le t\le T}U^{\alpha}_{\theta}(u(\cdot,t);I_1)
\\
&\qquad\quad\qquad\quad\qquad\quad\qquad\quad
+\int_{0}^{T}U^{\alpha}_{\theta}(u(\cdot,t);I_1)\,\rd t
+\sqrt{d}\,\|U^{\alpha}_{\theta}(u(\cdot,t);I_1)\|_{L^2(0,T)}\Big),
\end{aligned}
\end{equation}
which yields the error bound in \eqref{MainII-c-dp} on $I_1.$

For the other sub-intervals $I_j$, $j=2,\dots,M$, the analysis is the same as in the proof of Theorem~\ref{MainEstiA}. We therefore skip the details, and obtain the same estimates.
This completes the proof.
\end{proof}

\subsection{Numerical results}\label{Subsec43}
We conclude this section with additional numerical results, complementing Figure~\ref{Figs_1}, to further demonstrate the optimal convergence behavior of the LDG scheme for the aforementioned singularities.
In all tests below, we take $\Omega=(0,1)$ and $T=1,$ and use the same time-stepping scheme as in Figure~\ref{Figs_1}.

For comparison, we first consider Theorem~\ref{MainEstiA} with the singular solution:  $u(x,t)=x^{\pi} e^{2+\sin(x)-t}$ with $m=\infty$ in  Figure~\ref{Figs4-1} (left), and  observe the optimal convergence order in $p$ as estimated in  Theorem~\ref{MainEstiA} (also see Figure~\ref{Figs_1} for the same convergence behavior). As noted in  Remark~\ref{sing-rmk}, one may lose the $(2\pi+1)$th-order for $d=0$, but only have an order of $\pi+m-1/2$ if $m\le \pi+3/2,$ likewise for $d\not=0.$  To illustrate this, we test the LDG scheme with the exact solution:
$$ u(x,t)={\mathcal I}_{0
+}^{\pi} H(\zeta-x)\,t=\frac{1}{\Gamma(\pi+1)}\big(x^{\pi}-(x-\zeta)_+^{\pi+1}\big) t,$$
where $H$ is  the Heaviside function and $ (x-\zeta)_+ := \max\{x-\zeta,\,0\}.$ Note that
$\CFrD{\pi}{0+} u= H(\zeta-x) t\in W^{1,1}(I_1)$ for $\zeta>0,$ so we have $u \in  {\mathcal U_{0+}^{\pi, m}}(I_1)$ with $m=1.$ Then we infer from Theorem \ref{MainEstiA} that the convergence order in $p$ is
${\mathcal O}(p^{-\pi-1/2})$ for $d=0$ and ${\mathcal O}(p^{-\pi-1/2})$ for $d\not=0,$ respectively, which is in a good agreement with the numerical results in Figure~\ref{Figs4-1} (right).
\begin{figure}[!ht]
	\begin{center}
		{~}\hspace*{-20pt}	\includegraphics[width=0.45\textwidth]{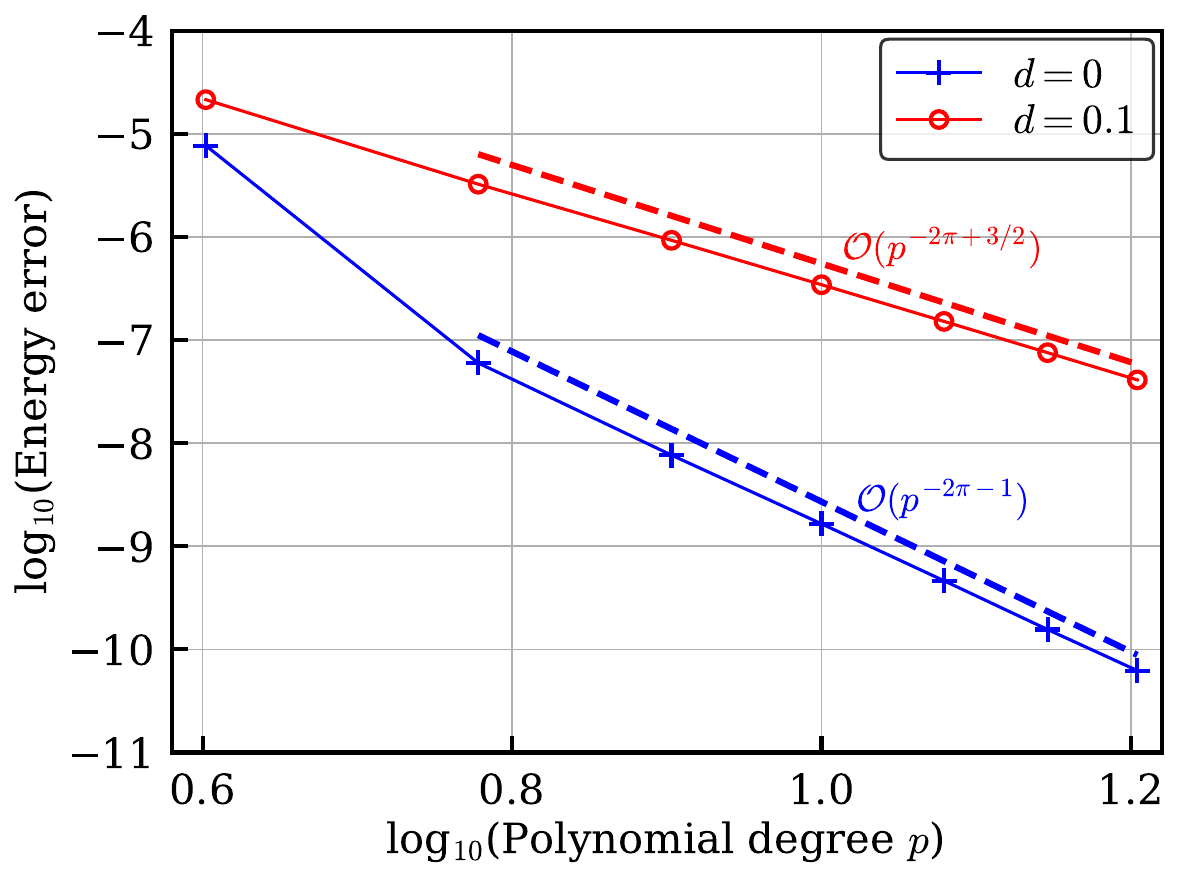}
		\hspace{10pt}
\includegraphics[width=0.45\textwidth]{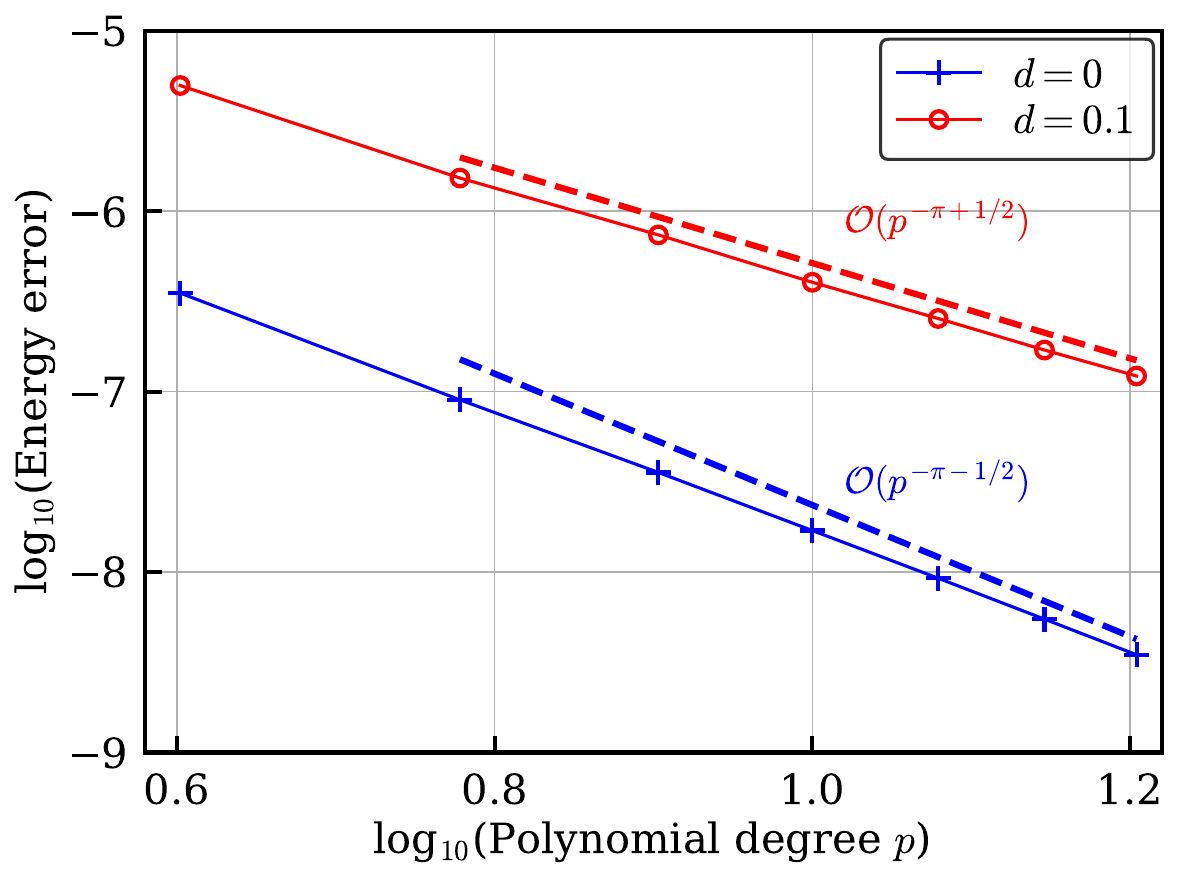}
	\caption{The convergence of the $p$-version for the non-smooth exact solutions, where $u(x,t)=x^{\pi} e^{2+\sin(x)-t}$ (left), and
    $u(x,t)={\mathcal I}_{0+}^{\pi} H(\zeta-x) \,t$ with  $\zeta = (x_0 + x_1)/2$ (right).
    The convection coefficient is $c=0.1$ and the diffusion coefficient is $d=0.1$ and $d=0.$
}
\label{Figs4-1}
	\end{center}
\end{figure}

We next conduct some numerical tests to validate Theorems~\ref{MainEstiB}--\ref{MainEstiC}. We choose the exact solution:  $u(x,t)=|x-\zeta|^{\pi}\, e^{2+\sin(x)-t}$ and consider (i) fitted case:
$\zeta=x_1$  and (ii) unfitted case:
 $\zeta= (x_0 + x_1)/2.$  We illustrate in
  Figure~\ref{Figs4-2} the convergence for both cases, which again perfectly match the predicted order. Indeed, the fitted case leads to a higher order than the unfitted case. Surprisingly, we observe from
  Figure~\ref{Figs4-1} (left) and  Figure~\ref{Figs4-2} (left) that there is a lost of half order for $d\not =0,$
  for the left-end point singularity and interior singularity on the grid.
  \begin{figure}[!ht]
	\begin{center}
		{~}\hspace*{-20pt}	 \includegraphics[width=0.45\textwidth]{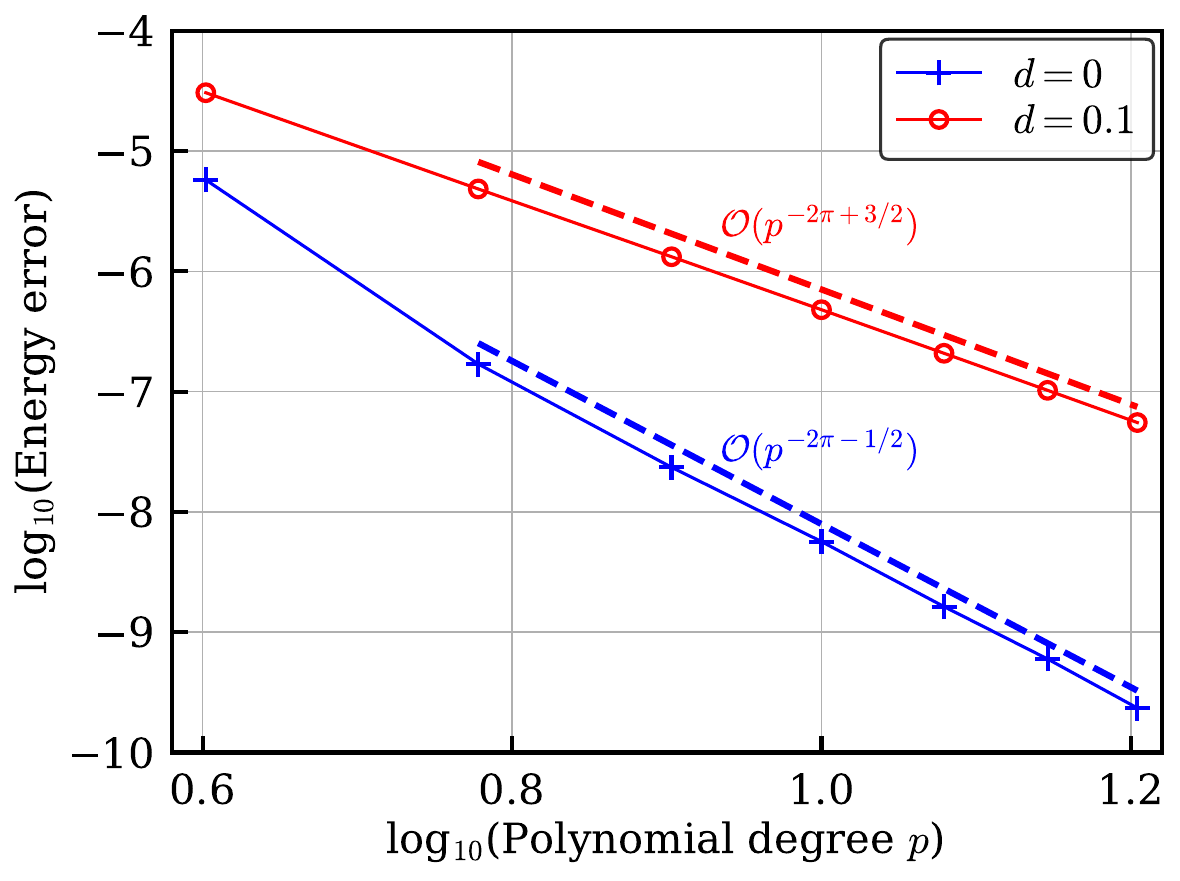}
		\hspace{10pt}
		\includegraphics[width=0.45\textwidth]{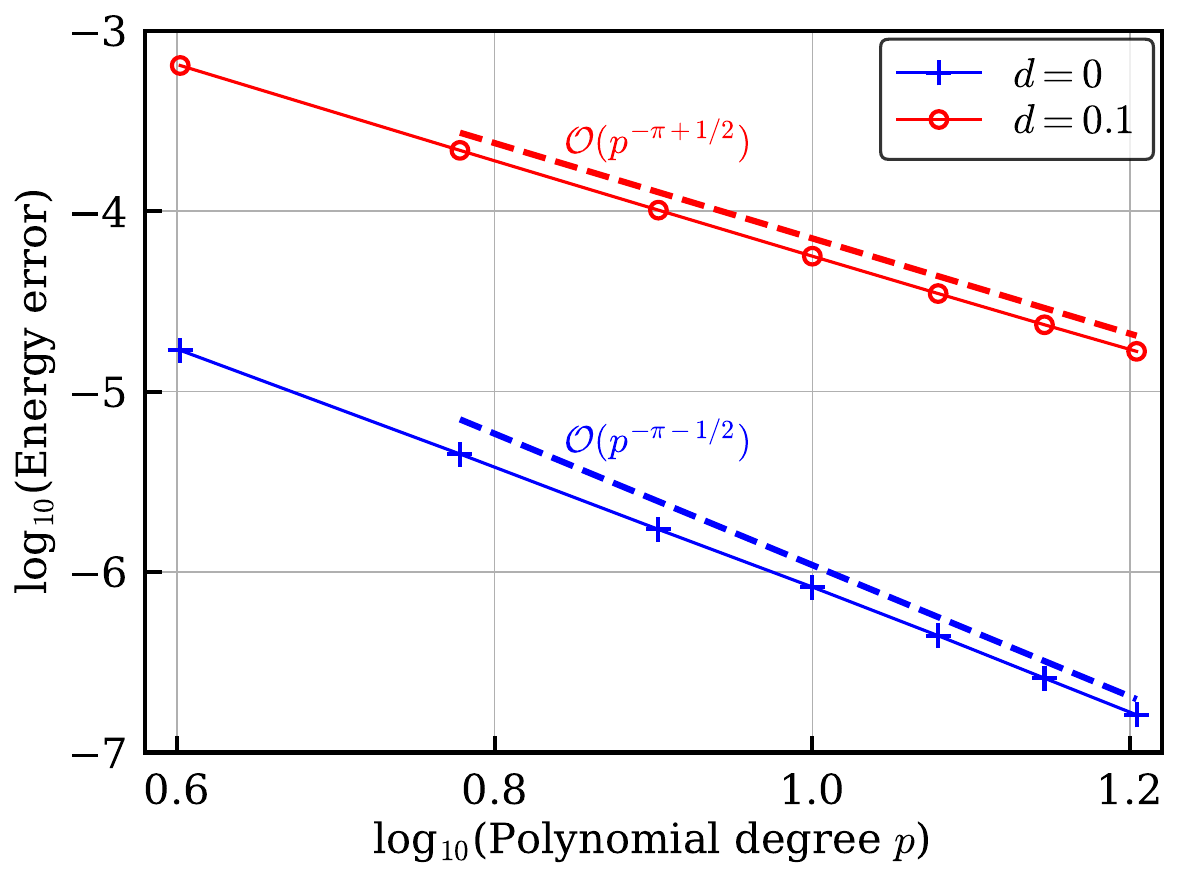}
	\caption{
    The convergence of the $p$-version for the nonsmooth exact solutions $u(x,t)=|x-\zeta|^{\pi} e^{2+\sin(x)-t}$, where $\zeta=x_1$  on the left side and $\zeta= (x_0 + x_1)/2$ on the right side.
    The convection coefficient is $c=0.1$ and the diffusion coefficient is $d=0.1$ and $d=0.$
}
\label{Figs4-2}
	\end{center}
\end{figure}

\section{Concluding remarks}
\label{ConcludingRem}
%In this paper, we have obtained optimal error estimates for the $p$-version of the LDG method for  a one-dimensional convection-diffusion equation.
%We derived  new fractional spaces to deal with the nonsmooth solutions.
%The optimal estimates of Gauss–Radau projections $\pi^+$ and $\pi^-$ with interior or endpoint
%singularities in new fractional spaces were proved.
%Extensions to multidimensional case is challenging, which constitute of our future work.
%In this paper, we have obtained optimal error estimates for the $p$-version of the LDG method applied to a one-dimensional convection-diffusion equation. We derived new fractional spaces to handle nonsmooth solutions. The optimal estimates of Gauss--Radau projections $\pi^+$ and $\pi^-$ with interior or endpoint singularities in these new fractional spaces were proved. Extending these results to the multidimensional case remains challenging and constitutes a significant part of our future work.

\begin{table}[h!]
\centering
\renewcommand{\arraystretch}{1.15}
\caption{$p$--version rates for the \emph{dominant singular term} (fixed $h$).}
\label{tab:p_orders_singular}
\begin{tabular}{ccc}
\hline
\textbf{Singularities} & \textbf{$d=0$} & \textbf{$d\neq 0$} \\
\hline
Left-endpoint (Theorem~\ref{MainEstiA})
& $\min\{2\alpha+1,\ \alpha+m-\tfrac12\}$
& $\min\{2\alpha-\tfrac32,\ \alpha+m-\tfrac32\}$ \\

Fitted interior   (Theorem~\ref{MainEstiB})
& $\min\{2\alpha+\tfrac12,\ \alpha+m-\tfrac12\}$
& $\min\{2\alpha-\tfrac32,\ \alpha+m-\tfrac32\}$ \\

Unfitted interior  (Thorem~\ref{MainEstiC})
& $\alpha+\tfrac12$
& $\alpha-\tfrac12$ \\
\hline
\end{tabular}

\vspace{0.3em}
\begin{minipage}{0.95\linewidth}
\footnotesize
$\alpha$: singularity exponent in $u(x,t)=\mathcal{O}(|x-x_j|^\alpha)$ near $x_j$.\\
$m$: an additional regularity index.
% integer regularity index for $\CFrD{\alpha}{}u$ in the fractional spaces.
\end{minipage}
\end{table}

The $hp$ local discontinuous Galerkin (LDG) method of Castillo et al.~\cite{Castillo2002MC} is an efficient scheme for convection–diffusion equations, but the available analysis for nonsmooth solutions predicts a loss of one order in the $p$–convergence.
In this paper we revisit this issue in a one–dimensional setting and close the gap between theory and computation by deriving new approximation results for the associated Gauss–Radau projections in fractional Sobolev spaces.
Within this framework we analyze the Legendre polynomial expansion of singular solutions and obtain sharp projection estimates, which in turn yield $p$–optimal a priori LDG error bounds for left–endpoint singularities and distinct $p$–rates in the fitted and unfitted interior cases, depending on whether the singular point lies at a mesh node or strictly inside an element.

The resulting $p$-version convergence orders for the dominant singular term (for fixed $h$) are summarized in Table~\ref{tab:p_orders_singular}.
For both the purely convective case $(d=0)$ and the convection--diffusion case $(d\neq 0)$, the table reports the attainable order in $p$ as a function of the singularity exponent $\alpha$, the additional regularity index $m$, and the location of the singular point.
The numerical experiments in Subsection~\ref{Subsec43} confirm these rates and demonstrate full consistency between the theoretical analysis and the observed LDG errors.
The approximation framework developed here arises naturally from the structure of Gauss–Radau projections and provides a useful tool for studying singular solutions in DG methods, shedding light on how $p$–suboptimality reported for other DG schemes may be resolved in more general settings.

 \end{document}